\documentclass[onefignum,onetabnum]{siamart220329}
\usepackage[utf8]{inputenc}
\usepackage[T1]{fontenc}
\usepackage{amsmath}
\usepackage{amsfonts}
\usepackage{subcaption}
\usepackage{amssymb}
\usepackage{makeidx}
\usepackage{algorithm}
\usepackage{makecell}
\usepackage{algpseudocode}
\usepackage{color}
\usepackage{graphicx}
\usepackage[left=2.5cm,right=2cm,top=2cm,bottom=2.5cm]{geometry}
\newtheorem{remark}{Remark}

\newcommand{\mol}{\omega}

\newtheorem{assumption}{Assumption}
\title{Minimal time nonlinear control via semi-infinite programming}
\author{Antoine Oustry\textsuperscript{1,2}, Matteo Tacchi\textsuperscript{3,4}}

\begin{document}
\stepcounter{footnote}
\footnotetext{Ecole des Ponts, Marne-la-Vallée, France.}
\stepcounter{footnote}
\footnotetext{Laboratoire d'informatique de l'\'Ecole polytechnique, Institut Polytechnique de Paris, Palaiseau, France.}
\stepcounter{footnote}
\footnotetext{Univ. Grenoble Alpes, CNRS, Grenoble INP (Institute of Engineering Univ. Grenoble Alpes), GIPSA-lab, 38000 Grenoble, France.}
\stepcounter{footnote}
\footnotetext{Corresponding author. \texttt{\href{mailto:matteo.tacchi@gipsa-lab.fr}{matteo.tacchi@gipsa-lab.fr}}}

\maketitle
\begin{abstract}
We address the problem of computing a control for a time-dependent nonlinear system to reach a target set in a minimal time. To solve this minimal time control problem, we introduce a hierarchy of linear semi-infinite programs, the values of which converge to the value of the control problem. These semi-infinite programs are increasing restrictions of the dual of the nonlinear control problem, which is a maximization problem over the subsolutions of the Hamilton-Jacobi-Bellman (HJB) equation. Our approach is compatible with generic dynamical systems and state constraints. Specifically, we use an oracle that, for a given differentiable function, returns a point at which the function violates the HJB inequality. We solve the semi-infinite programs using a classical convex optimization algorithm with a convergence rate of $O(\frac{1}{k})$, where $k$ is the number of calls to the oracle. This algorithm yields subsolutions of the HJB equation that approximate the value function and provide a lower bound on the optimal time. We study the closed-loop control built on the obtained approximate value functions, and we give theoretical guarantees on its performance depending on the approximation error for the value function. We show promising numerical results for three non-polynomial systems with up to $6$ state variables and $5$ control variables.
\end{abstract}


\section{Introduction}

\subsection{Motivation and related works} This paper deals with the control of a deterministic
dynamical system to reach a target set in a minimal time. We consider a general case of a time-dependent nonlinear system under nonlinear state constraints.  Several applications in various fields, such as robotics \cite{jazar_theory_2010}, aerospace \cite{trelat_optimal_2012}, maritime routing \cite{mannarini_graph-search_2020} or medicine \cite{zabi_time-optimal_2017}, can be formulated as minimal time control problems. Minimal time control, also known as time optimal control,  can be seen as a special case of the general framework of Optimal Control Problems (OCP). Solving an OCP for such generic dynamics and constraints is a difficult challenge, although deep theoretical tools are available such as the Pontryagin Maximum Principle (PMP)  \cite{bourdin_pontryagin_2015,clarke_relationship_1987, pontryagin_mathematical_1987} and the Hamilton-Jacobi-Bellman (HJB) equation \cite{crandall_viscosity_1983,frankowska_optimal_1989}. Those theoretical tools, initially developed in the unconstrained setting, have been extended to the case of state constraints \cite{capuzzo-dolcetta_hamilton-jacobi_nodate,soner_optimal_1986}. From an numerical point of view, the \textit{multiple shooting} techniques \cite{pesch_practical_1996,von_stryk_direct_1992} are based on the PMP, and reduce to the solution of a two-point boundary value problem. The \textit{direct methods} reduce to the solution a nonlinear programming problem after discretizing the time space, or parameterizing the control $u(t)$ in a finite dimensional subspace \cite{trelat_optimal_2012,von_stryk_direct_1992}. The celebrated Model Predictive Control (MPC) approach belongs to the category of direct methods \cite{camacho_model_2013}. Another approach is to compute the value function of the problem as a maximal subsolution of the HJB equation \cite{vinter_convex_1993}. This approach is related to the weak formulation of the OCP, which is an infinite dimensional linear program (LP) involving occupation measures. The dual problem of this LP is exactly the problem of finding a maximal subsolution of the HJB equation \cite{hernandez-hernandez_linear_1996,lasserre_nonlinear_2008}. In \cite{henrion_linear_2014,lasserre_nonlinear_2008}, the Moment Sum-of-Squares (SoS) hierarchy is used to approximate the solution of the resulting infinite dimensional LPs, in the case where the dynamics and the constraints of the OCP are defined by polynomials. The convergence rate of this numerical scheme is studied in \cite{korda_convergence_2017} for infinite-time discounted polynomial control problems. Still in the context of polynomial control problems, a work \cite{jones_polynomial_2023} based on the dual LP and the SoS hierarchy also studies the design of a closed-loop controller based on the approximate value function that is computed. In \cite{berthier2022infinite}, an extension of the SoS hierarchy based on Kernel methods is employed to extend this computation to general nonlinear system.  Regarding the methods specifically dedicated to time-optimal control, we find the same categories: direct methods such as MPC \cite{verschueren_stabilizing_2017}, indirect methods based on the PMP and  the bang-bang property \cite{liberzon_calculus_2012,olsder_time-optimal_1975} or methods based on convex optimization \cite{leomanni_time-optimal_2022}.

\subsection{Contribution} In this paper, we focus on the problem of  computing a control to reach a target set in a minimal time. We follow the line of works that use convex optimization to solve the dual problem of the nonlinear control problem, over the subsolutions of the HJB equation \cite{henrion_linear_2014,hernandez-hernandez_linear_1996,korda_convergence_2017,lasserre_nonlinear_2008,vinter_convex_1993}. In contrast to several works using the Moment-SoS  hierarchy \cite{jones_polynomial_2023,korda_convergence_2017,lasserre_nonlinear_2008,oustry_inner_2019,sager_efficient_2015}, the dynamical system and the state constraints considered here are generic and, in particular, are not assumed to be defined by polynomials. Instead of using polynomial optimization theory and the associated positivity certificates, our approach relies on the existence of a \textit{separation oracle} capable of returning, for a given differentiable function V, a point $(t, x)$ where the function $V$ does not satisfy the HJB inequality. Such an oracle can be provided by a global optimization solver or by a sampling scheme in a black-box optimization approach. In particular, our approach is compatible with the sampled-data control paradigm \cite{berthier2022infinite,bourdin_pontryagin_2015,korda_computing_2020}. Our contribution is manifold 
\begin{itemize}
\item We introduce a hierarchy of linear semi-infinite programs, the values of which converge to the value of the control problem. After regularization, we solve these semi-infinite programs using a classical algorithm with a convergence rate in $O(\frac{1}{k})$, where $k$ is the number of calls to the oracle. This yields subsolutions of the HJB equation that lower-approximate the value function and provide a certified lower bound on the minimum time.
\item  It is known that one can leverage any function $V(t,x)$ approximating the value function, to design a closed-loop, \textit{i.e.}, feedback controller \cite{henrion_nonlinear_2008,jones_polynomial_2023}. In this paper, we study the existence of trajectories generated by such a controller.
\item We study the performance of such a closed-loop controller, depending on how well $V(t,x)$ approximates the value function, in a way distinct from the analysis in \cite{jones_polynomial_2023}. In particular, this novel analysis enables us to give a sufficient condition for the closed-loop controller to effectively generate a trajectory reaching the target set within the considered time horizon.
\item We perform numerical experiments on three non-polynomial controlled systems and compute lower and upper bounds on the minimum time.
\end{itemize}

\subsection{Mathematical notation}
For any $p \in \mathbb{N}^*$, and $k\in \mathbb{N} \cup \{ \infty\}$, we denote by $C^k(\mathbb{R}^p) = C^k(\mathbb{R}^p,\mathbb{R})$ the vector space of real-valued functions with $k$ continuous derivatives over $\mathbb{R}^p$. For a given set $A \subset  \mathbb{R}^p$, for any function $f \in C^k(\mathbb{R}^p)$, we denote by $f_{|A}$ the restriction of $f$ on $A$; moreover, we define the vector space $C^k(\mathbb{R}^p | A) = \{ f_{|A} \colon f \in C^k(\mathbb{R}^p) \}$ of restrictions on $A$ of $C^k$ functions. For any locally Lipschitz function $f$, we denote by $\partial^c f $ its Clarke subdifferential \cite{clarke_generalized_1975}, to be distinguished from $\partial_{x_i} g $, the partial derivative of a differentiable function $g$ with respect to $x_i$.

For any two Lebesgue integrable functions $f,g \in L^1(\mathbb{R}^p)$, we define the convolution product $f \star g = g \star f$ as $f \star g(x) = \int_{\mathbb{R}^p} f(x) g(x-h) dh$. We emphasize that this convolution product is also well defined if $f$ is supported on a compact set, and $g$ is locally integrable. We denote by $\mathbb{R}[x_1, \dots x_p]$ the vector space of real multivariate polynomials with variables $x_1, \dots, x_p$, and $\mathbb{R}_d[x_1, \dots x_p]$ the vector space of such real multivariate polynomials with degree at most $d$.

For any set $A \subset \mathbb{R}^p$, we write $\mathsf{conv}(A)$ for the convex hull of the set $A$. For any nonempty set $A$, and any $x \in \mathbb{R}^p$, we denote $d(x,A) =\inf_{a \in A} \lVert x - a \rVert_2$ the distance between  the set $A$ and the point $x$. We also define the contingent cone to $A$ at $x \in A$, denoted $T_A(x)$ as the set of directions  $d \in \mathbb{R}^p$, such that there exist a sequence $(t_k) \in \mathbb{R}_{++}^\mathbb{N}$, and a sequence  $(d_k) \in (\mathbb{R}^p)^\mathbb{N}$, satisfying $t_k \rightarrow 0$,  $d_k \rightarrow d$, and $x + t_k d_k \in A$, for all $k \in \mathbb{N}$. Finally, we say that a property $P$ holds  ``almost everywhere'' (a.e.) on $A$, or equivalently ``for almost all $x \in A$'', to denote that there exists a set $N$ of Lebesgue measure zero such that the property $P$ holds for all $x \in A \setminus N$.

\section{Problem statement and Linear Programming formulations}

\subsection{Definition of the minimal time control problem}
Let $n$ and $m$ be nonzero integers. We consider on $\mathbb{R}^n$ the control system
\begin{align}
\dot{x}(t) = f(t,x(t), u(t)), \label{eq:system}
\end{align}
where $f\colon \mathbb{R} \times \mathbb{R}^{n} \times \mathbb{R}^m \to \mathbb{R}^n$ is Lipschitz continuous, and where the controls are bounded measurable functions, defined on intervals $[t_0, t_1] \subset [0, T]$, and taking their values in a compact set $U$ of $\mathbb{R}^m$. Let $X$ and $K\subset X$ be compact sets of $\mathbb{R}^n$ and $x_0 \in \mathbb{R}$. For $t_0, t_1 \geq 0$, a control $u$ is said admissible on $[t_0, t_1]$ whenever the solution $x(.)$ of  \eqref{eq:system}, such that $x(t_0) = x_0$, is well defined on $[t_0, t_1]$ and satisfies the constraints
\begin{align}
(x(t),u(t)) \in X \times U, \quad \text{a. e.  on } [t_0, t_1], \label{eq:constraints}
\end{align}
and satisfies the terminal state constraint
\begin{align}
x(t_1) \in K. 
\end{align}
We denote by $\mathcal{U}(t_0,t_1,x_0)$ the set of admissible controls on $[t_0, t_1]$. We consider the question of the minimal time problem from $x_0$ to $K$, 
\begin{align}
V^* (t_0,x_0) = \underset{\begin{subarray}{c} t_1 \in [t_0, T] \\ u(\cdot) \in \mathcal{U}(t_0,t_1,x_0)\end{subarray}}{\inf} t_1-t_0.
\label{eq:inf}
\end{align}
This is a particular case of the OCP with free final time \cite{lasserre_nonlinear_2008}, associated with the cost $\int_{t_0}^{t_1} \ell(t,x(t),u(t)) dt$ for $\ell(t,x(t),u(t)) = 1$. The function $V^*$ is called the value function of this minimal time control problem: this describes the smallest time to reach the target set $K$, starting from $x_0$ at time $t_0$.

\begin{assumption}
For any $(t,x) \in [0,T] \times X$, the set $f(t,x,U)$ is convex.
\label{as:convexF}
\end{assumption}
We underline that we do not have any convexity assumption on the constraint set $X$ and on the target set $K$. 
\begin{remark}
Even if the dynamical system of interest does not satisfy Assumption~\ref{as:convexF}, we can apply the present analysis to the \textit{convexified inclusion} $\dot{x}(t) \in \mathsf{conv} \: f(t,x(t), U)$. According to the Filippov-Wa\.{z}ewski relaxation Theorem \cite[Th.~10.4.4]{aubin_set-valued_2009}, the trajectories of the original control problem are dense in the set of trajectories of the convexified inclusion. The trajectories of the convexified inclusion may be seen as the limit of chattering trajectories, \textit{i.e.}, when the control oscillates infinitely fast and where the constraint set is infinitesimally dilated. 
\end{remark}

\begin{theorem}
Under Assumption~\ref{as:convexF}, the minimal time control problem \eqref{eq:system}-\eqref{eq:inf} associated with a starting point $(t_0, x_0) \in [0, T] \times X$ is either infeasible or admits an optimal trajectory. \label{th:exopti}
\end{theorem} 
\begin{proof}
We consider the case where a feasible trajectory exists. This is a direct application of \cite[Th.~2.1]{vinter_convex_1993}, which, among others, characterizes the existence of an optimal trajectory for a control problem over a differential inclusion. To emphasize the correspondence with the notation of \cite{vinter_convex_1993}, we highlight that we apply the theorem with: the running cost function $\ell(t,x,p) = 1$, the terminal cost function $g(t,x) = 0$, the set-valued map $F(t,x) = f(t,x, U)$, the constraint set $A = [0,T] \times X$ and the target set $C = [0,T] \times K$. We underline that the assumptions (H1)-(H5) in \cite{vinter_convex_1993} are satisfied here; more precisely, we highlight that our Assumption~\ref{as:convexF} enforces (H2) and the hypothesis that a feasible trajectory exists enforces (H4).
\end{proof}

\subsection{Hamilton-Jacobi-Bellman equation and subsolutions}
In optimal control theory, a well-known sufficient condition for a function $V$ to be the value function $V^*$ is to satisfy the Hamilton-Jacobi-Bellman (HJB) Partial Differential Equation (PDE). This PDE may be seen as a continuous time generalization of Bellman's dynamic programming optimality principle in discrete time \cite{bellman_dynamic_1966}. In our minimal time control setting, the HJB PDE reads
\begin{align}
\partial_t V (t,x) + \min_{u \in U} \{ 1 + \nabla_x V(t,x)^\top f(t,x,u) \} = 0, \quad \forall (t,x) \in [0,T] \times X  \label{eq:hjb1}\\
V(t,x) = 0, \quad \forall (t,x) \in [0,T] \times K. \label{eq:hjb2}
\end{align}
In general, differentiable solutions of this PDE may not exist, so the concept of viscosity solutions is typically used \cite{crandall_viscosity_1983}. Another approach to get around the lack of a differentiable solution to the HJB PDE consists in leveraging the concept of subsolutions \cite{vinter_convex_1993}, \textit{i.e.}, functions $V \in C^1(\mathbb{R}^{n+1})$ satisfying the following inequalities:
\begin{align}
\partial_t V (t,x) + \min_{u \in U} \{ 1 + \nabla_x V(t,x)^\top f(t,x,u) \} \geq 0, \quad \forall (t,x) \in [0,T] \times X  \label{eq:subsol1}\\
V(t,x) \leq 0, \quad \forall (t,x) \in [0,T] \times K. \label{eq:subsol2}
\end{align}
The following lemma states that any subsolution of the HJB PDE is an under-approximation of the value function.
\begin{lemma}
For any $V \in C^1(\mathbb{R}^{n+1})$ satisfying Eqs.~\eqref{eq:subsol1}-\eqref{eq:subsol2}, the following holds:
\begin{align*}
V(t,x) \leq V^*(t,x), \quad \forall (t,x) \in [0,T] \times X.
\end{align*}
\label{lem:subsol}
\end{lemma}
\begin{proof} We take any $(t,x) \in [0,T] \times X$ and we consider that $V^*(t,x)< \infty$, since the case $V^*(t,x) = \infty$ is trivial. Hence, according to Th.~\ref{th:exopti}, there exists an admissible control $\Bar{u}(\cdot) \in \mathcal{U}(t,t_1, x)$ for $t_1 \in [t, T]$ such that  $V^*(t,x) = t_1 - t$, and an associate trajectory $\Bar{x}(t)$ such that $\Bar{x}(t) = x$ and $\Bar{x}(t_1) \in K$. We observe that $\frac{d}{dt} [V(t,\Bar{x}(t))] = \partial_t V(t,\Bar{x}(t)) + \nabla_x V (t, \Bar{x}(t))^\top f(t,\Bar{x}(t),\Bar{u}(t)) \geq -1$ a.e. on $[t,t_1]$, the inequality holding since  $V$ satisfies Eq.~\eqref{eq:subsol1}. By integration, we observe that $V(t_1,\Bar{x}(t_1)) - V(t,x) \geq t - t_1 = - V^*(t,x)$, \textit{i.e.}, $V(t_1,\Bar{x}(t_1)) + V^*(t,x) \geq V(t,x)$.   As $V$ satisfies Eq.~\eqref{eq:subsol2} and as $\Bar{x}(t_1) \in K$, we observe that $0 \geq V(t_1,\Bar{x}(t_1))$, and therefore $V^*(t,x) \geq V(t,x)$.
\end{proof}

\subsection{Infinite dimensional Linear Programming formulations}
In the rest of the paper, we consider a given point $x_0 \in X$, and we raise the issue of computing the minimal time from $x_0$ to $K$ and the associated control. We make the following assumption:
\begin{assumption}
There exists an admissible control $u \in \mathcal{U}(0,t_1,x_0)$ associated with $t_1 \in [0, T]$. In other words, $V^*(0,x_0) \leq t_1 < \infty$.
\label{as:finite}
\end{assumption}
We consider the optimization problem of finding the subsolution of the HJB PDE that maximizes the evaluation in $(0, x_0)$. This problem may be cast as an infinite dimensional linear program:
\begin{align}
\begin{array}{rll}
\underset{V \in \mathcal{F}} \sup & V(0,x_0) & \\
\text{s.t.} & \partial_t V (t,x) + 1 +  \nabla_x V(t,x)^\top f(t,x,u) \geq 0 & \forall (t,x,u) \in [0,T] \times X \times U \\
& V(t,x) \leq 0 &  \forall (t,x) \in [0,T] \times K,
\end{array}
\tag{\mbox{$D_\mathcal{F}$}}
\label{eq:lpc1}
\end{align}
with $\mathcal{F} \in \{ C^1(\mathbb{R}^{n+1}), C^\infty(\mathbb{R}^{n+1}),\mathbb{R}[t,x_1,\dots, x_n] \}$. For a given $V \in \mathcal{F}$, the feasibility in \eqref{eq:lpc1} is clearly equivalent to the satisfaction of Eqs.~\eqref{eq:subsol1}-\eqref{eq:subsol2}. We also note that this infinite dimensional LP formulation corresponds to the dual LP formulation in \cite{lasserre_nonlinear_2008}; in fact, this is the dual problem of an infinite dimensional LP formulation of the control problem based on occupation measures. According to the next theorem, the problem \eqref{eq:lpc1} on $C^1$ functions has the same value as the minimal time control problem.
\begin{theorem}
Under Assumption~\ref{as:convexF} and Assumption~\ref{as:finite}, and for $\mathcal{F} = C^1(\mathbb{R}^{n+1})$, the value of the LP formulation \eqref{eq:lpc1} equals $V^*(0,x_0)$.
\label{th:duality}
\end{theorem}
\begin{proof}
As for the proof of Th.~\ref{th:exopti}, this is a direct application of \cite[Th.~2.1]{vinter_convex_1993}, which also states the absence of duality gap between a control problem over a differential inclusion and a maximization problem over subsolutions of the HJB equation. We underline that the assumptions (H1)-(H5) in \cite{vinter_convex_1993} are satisfied here; more precisely, our Assumption~\ref{as:convexF} enforces (H2) and our Assumption~\ref{as:finite} enforces (H4).
\end{proof}
Theorem~\ref{th:smoothvf} extends this result by stating that we can require the subsolutions of the HJB equation to be in $C^\infty(\mathbb{R}^{n+1})$,  while preserving the value of \eqref{eq:lpc1}. Before stating this theorem, we introduce an auxiliary lemma.
\begin{lemma}
For any $V \in C^1(\mathbb{R}^{n+1})$ satisfying Eqs.~\eqref{eq:subsol1}-\eqref{eq:subsol2} with feasibility error less or equal than $\eta \geq 0$, $V(t,x) + \eta(t-1-T)$ satisfies Eqs.~\eqref{eq:subsol1}-\eqref{eq:subsol2}.\label{lem:feas}
\end{lemma}
\begin{proof}
We introduce $\tilde{V}(t,x) = V(t,x) + \eta(t-1-T)$. By assumption on $V(t,x)$, we have $\partial_t V (t,x) + 1 +  \nabla_x V(t,x)^\top f(t,x,u) \geq -\eta$, for all $(t,x,u) \in [0,T] \times X \times U$. By linearity, and since $\partial_t (t-1-T) = 1$ and $\nabla_x (t-1-T) = 0$,  $\partial_t \tilde{V} (t,x) + 1 +  \nabla_x \tilde{V}(t,x)^\top f(t,x,u) \geq 0$. By assumption on $V(t,x)$, we have $V(t,x) \leq \eta$, for all $(t,x) \in [0,T] \times K$. Hence, $\tilde{V}(t,x) \leq \eta + \eta(t-1-T) \leq \eta + \eta(T-1-T) = 0$ for all $(t,x) \in [0,T] \times K$.
\end{proof}
\begin{theorem}
Under Assumption~\ref{as:convexF} and Assumption~\ref{as:finite}, and for $\mathcal{F} = C^\infty(\mathbb{R}^{n+1})$, the value of the LP formulation \eqref{eq:lpc1} equals $V^*(0,x_0)$.
\label{th:smoothvf}
\end{theorem}
\begin{proof} We consider $\mathcal{F} = C^\infty(\mathbb{R}^{n+1})$ and we use the notation Y to denote the compact set $[0, T] \times X$. We fix $\epsilon > 0$, and we will prove that there exists $V \in C^\infty(\mathbb{R}^{n+1})$ that is feasible in \eqref{eq:lpc1} and such that $V(0,x_0) \geq V^*(0,x_0) - \epsilon$. According to Th.~\ref{th:duality}, there exists $V_1 \in C^1(\mathbb{R}^{n+1})$ that is feasible in \eqref{eq:lpc1} and such that $V_1(0,x_0) \geq V^*(0,x_0) - \frac{\epsilon}{2}$. For any $\sigma \in (0,1]$, we introduce the mollified function $V_{1 \sigma} = V_1 \ast \phi_\sigma \in C^\infty(\mathbb{R}^{n+1})$, where $\mol_\sigma$ is the standard mollifier defined as $\mol_\sigma(y) = \frac{1}{\sigma^{n+1}} \mol(y/\sigma)$, where $\mol(y) = \left \lbrace \begin{array}{ll}
\xi e^{-\frac{1}{1-\lVert y \rVert^2}} & \text{ if } \lVert y \rVert < 1 \\
0 & \text{ if } \lVert y \rVert \geq 1
\end{array} \right.$ for a given constant $\xi > 0$ such that $\int_{\mathbb{R}^{n+1}} \mol(y)dy = 1$. Hence, a simple change of variable shows that $\int_{\mathbb{R}^{n+1}} \mol_\sigma(y)dy = 1$. We also underline that $\mol_\sigma$ is non-negative and supported on the ball $B(0,\sigma)$. For any $y = (t,x) \in Y$ and any $\sigma \in (0,1]$, we have that $|V_1(y) - V_{1\sigma}(y)| = |V_1(y) - \int_{B(0,\sigma)} V_1(y - h) \mol_\sigma(h) dh | =   |\int_{B(0,\sigma)} (V_1(y) - V_1(y - h)) \mol_\sigma(h) dh |$ as $\int_{B(0,\sigma)} \mol_\sigma(h)dh = 1$. We denote by $L_V$ an upper bound for the continuous function $\rVert \nabla V_1(y) \lVert_2$ over the compact set $\hat{Y} = \{ y \in \mathbb{R}^{n+1} \colon d(y,Y) \leq 1 \}$, which is a Lipschitz constant for the function $V_1$. We deduce, by triangular inequality and non-negativity of $\mol_\sigma$ that for any $y \in Y$,
\begin{align}
|V_1(y) - V_{1\sigma}(y)| & \leq \int_{B(0,\sigma)} |V_1(y) - V_1(y - h)| \mol_\sigma(h) dh \\
& \leq \int_{B(0,\sigma)} L_V \lVert h \rVert \mol_\sigma(h) dh \\
& \leq L_V \sigma  \int_{B(0,\sigma)} \lVert h/\sigma \rVert \mol(h/\sigma) \frac{1}{\sigma^{n+1}} dh  \\
& \leq L_V \sigma \underbrace{\int_{B(0,1)} \lVert \tilde{h} \rVert \mol(\tilde{h}) d\tilde{h}.}_{\text{constant, denoted $\mathcal{I}$.}} \label{eq:boundingdiff}
\end{align}
By property of the mollifiers \cite{hormander_analysis_2003}, we have $\partial_i V_{1\sigma}(y) = \partial_i ( V_1 \ast \mol_\sigma) = (\partial_i V_1 \ast \mol_\sigma)$ for any $i \in \{t,x_1, \dots, x_n\}$. Therefore, $\partial_t V_{1\sigma} (y) =  \int_{B(0,\sigma)} \partial_t V_{1} (y-h)\mol_\sigma(h)dh$ and $\nabla_x V_{1\sigma}(y) = \int_{B(0,\sigma)} \nabla_x V_{1} (y-h)\mol_\sigma(h)dh$. Using the equality $\int_{B(0,\sigma)} \mol_\sigma(h)dh = 1$, we deduce that for any $y \in Y$,
\begin{align}
\hspace{-0.5cm} \partial_t V_{1\sigma} (y) + 1 +  (\nabla_x V_{1\sigma}(y))^\top f(y,u) &= \int_{B(0,\sigma)} (\partial_t V_{1} (y-h) + 1 +  (\nabla_x V_{1}(y-h))^\top f(y,u))\mol_\sigma(h)dh \label{eq:decfirst} \\
& = \int_{B(0,\sigma)} (\partial_t V_{1} (y-h) + 1 +  (\nabla_x V_{1}(y-h))^\top f(y-h,u))\mol_\sigma(h)dh  \\ & \quad \quad \quad + \int_{B(0,\sigma)} \nabla_x V_{1}(y-h)^\top(f(y,u) - f(y-h,u))\mol_\sigma(h)dh.  \label{eq:declast} 
\end{align}
We compute lower bounds for the two terms of the sum. We start with the second term:  using Cauchy-Schwarz inequality, we notice that $\int_{B(0,\sigma)} \nabla_x V_{1}(y-h)^\top(f(y,u) - f(y-h,u))\mol_\sigma(h)dh \geq - \int_{B(0,\sigma)} \lVert V_{1}(y-h) \rVert \lVert f(y,u) - f(y-h,u) \rVert \mol_\sigma(h)dh$. Noticing that $\lVert \nabla_x V_{1}(y-h) \rVert \leq L_V$, since $y-h \in \hat{Y}$ for any $h \in  B(0,\sigma) \subset B(0,1)$, and introducing the Lipschitz constant $L_f$ for $f$, we have 
\begin{align}
\int_{B(0,\sigma)} \nabla_x V_{1}(y-h)^\top(f(y,u) - f(y-h,u))\mol_\sigma(h)dh \geq  - L_V L_f \int_{B(0,\sigma)} \lVert h \rVert  \mol_\sigma(h)dh
= - L_V L_f \sigma \mathcal{I}. \label{eq:bound1}
\end{align}
We define $\eta = \frac{\epsilon}{2(T+2)}$. We introduce the compact set $Z = [0,T] \times X \times U$ and the family of compact sets $Z_\delta = \{ z \in  \mathbb{R}^N \: \colon \: d(z,Z)  \leq \delta \}$ for $\delta \in (0, 1]$. For any $z = (y,u) \in Z_1$, we introduce $\psi(z) = \partial_t V_{1} (y) + 1 +  (\nabla_x V_{1}(y))^\top f(y,u)$. The function $\psi(z)$ is continuous and according to Lemma~\ref{lem:valuefunction}, there exists $\sigma_1 > 0$ such that $\min_{z \in Z_{\sigma}} \psi(z)  \geq \min_{z \in Z}  \psi(z) - \frac{\eta}{2}$ for any $\sigma \in (0, \sigma_1]$. By feasibility of $V_1$ in \eqref{eq:lpc1}, we know that $\min_{z \in Z}  \psi(z) \geq 0$, which yields that $\psi(z) \geq - \frac{\eta}{2}$ for any $z \in Z_{\sigma_1}$. We deduce that 
\begin{align}
\int_{B(0,\sigma)} (\partial_t V_{1} (y-h) + 1 +  (\nabla_x V_{1}(y-h))^\top f(y-h,u))\mol_\sigma(h)dh \geq - \int_{B(0,\sigma)}  \frac{\eta}{2} \mol_\sigma(h)dh = - \frac{\eta}{2}, \label{eq:bound2}
\end{align}
since $(y-h,u) \in Z_\sigma$ for any $h \in B(0,\sigma)$. Combining the decomposition of Eqs.~\eqref{eq:decfirst}-\eqref{eq:declast}, with the lower bounds of Eq.~\eqref{eq:bound1} and \eqref{eq:bound2}, we deduce that
\begin{align}
\partial_t V_{1\sigma} (y) + 1 +  (\nabla_x V_{1\sigma}(y))^\top f(y,u) \geq - (L_V L_f \sigma \mathcal{I} + \frac{\eta}{2}), \label{eq:boundinghjb}
\end{align}
for any $(y,u) = (t,x,u) \in [0,T] \times X \times U$ and $\sigma \in (0,\sigma_1]$. We define $\tilde{\sigma} = \min\{\sigma_1, \frac{\eta}{2 L_V L_f \mathcal{I}},  \frac{\eta}{L_V \mathcal{I}} \}$. From Eq.~\eqref{eq:boundingdiff} and Eq.~\eqref{eq:boundinghjb}, we deduce that 
\begin{align}
V_{1 \tilde{\sigma}}(0,x_0) \geq V_1(0,x_0) - \eta \geq V^*(0,x_0) - \frac{\epsilon}{2} - \eta  & \label{eq:boundobj}\\
V_{1 \tilde{\sigma}}(t,x) \leq V_{1}(t,x) + \eta \leq \eta&, \quad \forall (t,x) \in [0,T] \times K  \label{eq:proofinal1}\\
\partial_t V_{1\tilde{\sigma}} (t,x) + 1 +  \nabla_x V_{1\tilde{\sigma}}(t,x)^\top f(t,x,u) \geq -\eta &, \quad \forall (t,x,u) \in [0,T] \times X \times U. \label{eq:proofinal2}
\end{align}
From Lemma~\ref{lem:feas}, we deduce that $V(t,x) =  V_{1 \tilde{\sigma}}(t,x) + \eta(t-1-T) \in C^\infty(\mathbb{R}^{n+1})$ is feasible in \eqref{eq:lpc1}. From Eq.~\eqref{eq:boundobj}, we deduce that $V(0,x_0) \geq V^*(0,x_0) - \frac{\epsilon}{2} - \eta - (1+T)\eta$, and by definition of $\eta$, $V(0,x_0) \geq V^*(0,x_0) - \epsilon$. 
\end{proof}
The next theorem underlies the convergence proof of the hierarchy of semi-infinite problems in Sect.~\ref{sec:convexsipsubsolutions}: if we restrict to polynomials HJB subsolutions, the value of the problem \eqref{eq:lpc1} remains unchanged.
\begin{theorem}
Under Assumption~\ref{as:convexF} and Assumption~\ref{as:finite}, and for $\mathcal{F} = \mathbb{R}[t,x_1, \dots, x_n]$, the value of the LP formulation \eqref{eq:lpc1} equals $V^*(0,x_0)$.  
\label{th:polyvf}
\end{theorem}
\begin{proof}
We consider $\mathcal{F} = \mathbb{R}[t,x_1, \dots, x_n]$. For a given $\epsilon > 0$, and we will prove that there exists $V \in \mathbb{R}[t,x_1, \dots, x_n]$ that is feasible in \eqref{eq:lpc1} and such that $V(0,x_0) \geq V^*(0,x_0) - \epsilon$. According to Th.~\ref{th:smoothvf}, there exists a function $Q \in C^\infty(\mathbb{R}^{n+1})$ which is a subsolution of the HJB equation and such that $Q(0,x_0) \geq V^*(0,x_0) - \frac{\epsilon}{2}$. We notice that $Q$ has a locally Lipschitz gradient. Therefore, we can apply Lemma~\ref{lem:polapprox}. This yields, in particular, that for any $\nu > 0$, there exists a polynomial $w \in \mathbb{R}[t,x_1, \dots, x_n]$ such that  for all $(t,x) \in [0,T] \times X$, $| w (y) -  w (y) | \leq \nu$ and $|\partial_i w (t,x) -  \partial_i V (t,x) | \leq \nu, \: i \in \{t,x_1, \dots, x_N \}$. We deduce that $|\partial_t Q (t,x) + \nabla_x Q(t,x)^\top f(t,x,u) - \partial_t w (t,x) + \nabla_x w(t,x)^\top f(t,x,u) | \leq |\partial_t Q (t,x) - \partial_t w (t,x) | + \sum_{i=1}^n  |\partial_{x_i} Q (t,x) - \partial_{x_i} w (t,x) | M_i \leq \nu (1 + \sum_{i=1}^n M_i)$, where $M_i = \max_{(t,x,u) \in [0,T] \times X \times U} |f_i(t,x,u)|$. Therefore, we observe that for all $(t,x,u) \in [0,T] \times X \times U$,
\begin{align}
\partial_t w (t,x) + 1 + \nabla_x w (t,x)^\top f(t,x,u) & \geq  \partial_t Q (t,x) + 1 + \nabla_x Q (t,x)^\top f(t,x,u) - \nu (1 + \sum_{i=1}^n M_i) \\
& \geq - \nu (1 + \sum_{i=1}^n M_i),
\label{eq:hjbpol}
\end{align}
as $Q$ is a subsolution of the HJB equation. In summary, for $\nu = \eta(1 + \sum_{i=1}^n M_i)^{-1} \leq \eta$,
\begin{align}
w(0,x_0) \geq Q(0,x_0) - \eta \geq V^*(0,x_0) - \frac{\epsilon}{2} - \eta \label{eq:firstpoly}\\
w(t,x) \leq Q(t,x) + \eta \leq \eta &, \quad \forall (t,x) \in [0,T] \times K \label{eq:secpoly} \\
\partial_t w (t,x) + 1 + \nabla_x w (t,x)^\top f(t,x,u) \geq -\eta&, \quad \forall (t,x,u) \in [0,T] \times X \times U \label{eq:lastpoly},
\end{align}
the last inequality following from Eq.~\eqref{eq:hjbpol}. Based on Eqs.~\eqref{eq:secpoly}-\eqref{eq:lastpoly} and Lemma~\ref{lem:feas}, we notice that the polynomial $V(t,x) =  w(t,x) + \eta(t-T-1)\in \mathbb{R}[t,x_1, \dots, x_n]$ is feasible in the problem \eqref{eq:lpc1}. Having defined $\eta = \frac{\epsilon}{2(T+2)}$, we see, based on Eq.~\eqref{eq:firstpoly},  that it satisfies $V(0,x_0) \geq  V^*(0,x_0) - \frac{\epsilon}{2} - \eta - \eta(T+1) = V^*(0,x_0) - \epsilon$.
\end{proof}

\section{Convex semi-infinite programming to compute near-optimal subsolutions}
\label{sec:convexsipsubsolutions}
For $\mathcal{F}$ being either $C^1(\mathbb{R}^{n+1})$, $C^\infty(\mathbb{R}^{n+1})$, or $\mathbb{R}[t,x_1,\dots, x_n] \}$, the linear program \eqref{eq:lpc1} is infinite dimensional, and thus, not tractable as it stands. Therefore, we next present a hierarchy of convex SIP problems that are solvable with a dedicated algorithm, to compute subsolutions to the HJB equation that are near optimal in the problem \eqref{eq:lpc1}.

\subsection{A hierarchy of linear semi-infinite programs}
Instead of having an optimization space $\mathcal{F}$ that is infinite dimensional, we suggest to restrict to the finite dimensional subspaces $\mathbb{R}_d[t,x_1,\dots, x_n]$ of polynomials of degree bounded by $d$. This restricted dual problem is:
\begin{align}
\begin{array}{rll}
\underset{V \in \mathbb{R}_d[t,x_1,\dots, x_n]} \sup & V(0,x_0) & \\
\text{s.t.} & \partial_t V (t,x) + 1 +  \nabla_x V(t,x)^\top f(t,x,u) \geq 0 & \forall (t,x,u) \in [0,T] \times X \times U \\
& V(t,x) \leq 0 &  \forall (t,x) \in [0,T] \times K.
\end{array}
\tag{\mbox{$R_d$}}
\label{eq:lprd}
\end{align}
In the rest of the paper, we will denote by $N$ the dimension of the vector space $\mathbb{R}_d[t,x_1,\dots, x_n]$ and  $\Phi(t,x) \in \mathbb{R}^{N}$ a basis of this space. For both objects, there is indeed a dependence of $d$, that is implicit here for readability reasons.  For any $V \in \mathbb{R}_d[t,x_1,\dots, x_n]$, we introduce the vector $\theta$ of the coordinates of $V$ in the basis $\Phi$. Hence, we have the relation
\begin{align}
V(t,x) = \theta^\top \Phi(t,x) \in \mathbb{R}_d[t,x_1,\dots, x_n].
\end{align}
Expressing problem~\eqref{eq:lprd} as an optimization problem over the vector of coefficients, it appears clearly that this is a linear semi-infinite program.
\begin{proposition}
For $d \in \mathbb{N}^*$, problem  \eqref{eq:lprd}
 is a linear semi-infinite program, \textit{i.e}, a linear program with a finite number of variables and an infinite number of constraints. More precisely, there exist a vector $c \in \mathbb{R}^N$, and a compact set $\mathcal{Y} \subset \mathbb{R}^{N+1}$ such that \eqref{eq:lprd} reads
\begin{align}
\begin{array}{rll}
\underset{\theta \in \mathbb{R}^N}{\sup} & c^\top \theta &  \\
\text{s.t.} & a^\top \theta + b \leq 0 & \forall (a,b) \in \mathcal{Y}.
\end{array}
\tag{${SIP}$}
\label{eq:sip}
\end{align}
\end{proposition}
\begin{proof}
We define the vector $c = \Phi(0,x_0)$, and the compact sets
\begin{align}
& \mathcal{Y}_1 = \{ (-\partial_t\Phi(t,x) - \nabla_x\Phi(t,x)^\top f(t,x,u), -1), (t,x,u) \in [0,T] \times X \times U \} \\
& \mathcal{Y}_2 = \{ (\Phi(t,x),0), (t,x) \in [0,T] \times K \} \\
& \mathcal{Y} =  \mathcal{Y}_1 \cup \mathcal{Y}_2.
\end{align}
We see that for any $V_\theta(t,x) = \theta^\top \Phi(t,x) \in \mathbb{R}_d[t,x_1,\dots, x_n]$, $V_\theta(0,x_0) = c^\top \theta$, and $V_\theta(t,x)$ is feasible in $\eqref{eq:lprd}$ if and only if $a^\top \theta + b \leq 0$, for all $(a,b) \in \mathcal{Y}$.
\end{proof}
We will see in the next section how to efficiently solve those semi-infinite programs. Prior to that, we state the convergence of this hierarchy of semi-infinite programs.
\begin{theorem}
The sequence $\mathsf{val}\eqref{eq:lprd}$ converges to $V^*(0,x_0)$ when $d \rightarrow \infty$.
\end{theorem}
\begin{proof}
On the one hand, we introduce the notation $v_d = \mathsf{val}\eqref{eq:lprd}$. This sequence is obviously an increasing sequence, bounded above by $V^*(0,x_0)$. Hence, it converges to a value $\ell$, and any subsequence converges to $\ell \leq V^*(0,x_0)$. On the other hand, Th.~\ref{th:polyvf} guarantees that there exists a sequence of polynomials $w_k \in \mathbb{R}[t,x_1,\dots,x_n]$ that are feasible in \eqref{eq:lpc1} and such that $w_k(0,x_0) \rightarrow_k V^*(0,x_0)$. By definition, we have $v_{d_k} \geq w_k(0,x_0)$, where $d_k = \mathsf{deg}(w_k)$. Up to the extraction of a subsequence of $(w_k)$, we can assume that the sequence $d_k$ increasing, therefore $ (v_{d_k})_{k\in\mathbb{N}}$ is a subsequence of $(v_d)_{d\in\mathbb{N}}$. As $v_{d_k} \rightarrow \ell$ and $w_k(0,x_0) \rightarrow_k V^*(0,x_0)$, we deduce that $\ell \geq V^*(0,x_0)$, which yields the equality  $\ell = V^*(0,x_0)$.
\end{proof}

\subsection{Regularization and solution of the semi-infinite programs}
We introduce a quadratic regularization in the semi-infinite program \eqref{eq:sip}, yielding the following formulation depending on $\mu \in \mathbb{R}_{++}$:
\begin{align}
\begin{array}{rll}
\underset{\theta \in \mathbb{R}^N}{\max} & c^\top \theta - \frac{\mu}{2} \lVert \theta \rVert^2 &  \\
\text{s.t.} & a^\top \theta + b \leq 0 & \forall (a,b) \in \mathcal{Y}.
\end{array}
\tag{\mbox{${SIP}_\mu$}}
\label{eq:sipr}
\end{align}

\begin{proposition}
For any $\mu > 0$, the semi-infinite program \eqref{eq:sipr} has a unique optimal solution with value $\mathsf{val}\eqref{eq:sipr} \leq V^*(0,x_0)$. Moreover, $\mathsf{val}\eqref{eq:sipr} \underset{\mu \rightarrow 0}{\rightarrow} \mathsf{val}\eqref{eq:sip}$.  
\end{proposition}
\begin{proof}
The feasible set of \eqref{eq:sipr} being convex, and the objective function being strongly concave, this optimization problem admits a unique maximum $\theta$. By definition, $\mathsf{val}\eqref{eq:sipr} = c^\top \theta - \frac{\mu}{2} \lVert \theta \rVert^2 \leq c^\top \theta \leq \mathsf{val}\eqref{eq:sip}$, since $\theta$ is also feasible in the maximization problem \eqref{eq:sip}. Additionally, $\mathsf{val}\eqref{eq:sip} \leq V^*(0,x_0)$, since any function $V$ feasible in $\eqref{eq:lprd}$ satisfies $V(0,x_0) \leq V^*(0,x_0)$. We also notice that the function $\mu \mapsto \mathsf{val}\eqref{eq:sipr}$ is decreasing, so it admits a limit $\ell$ at $0^+$, due to the aforementioned inequalities, $\ell \leq \mathsf{val}\eqref{eq:sip}$. For any $\mu, \epsilon > 0$, if we take $\theta_\epsilon$ an $\epsilon$-optimal solution in the problem \eqref{eq:sip}, we see that $\mathsf{val}\eqref{eq:sip} - \epsilon - \frac{\mu}{2} \lVert \theta_\epsilon \rVert^2 \leq  c^\top \theta_\epsilon - \frac{\mu}{2} \lVert \theta_\epsilon \rVert^2 \leq \mathsf{val}\eqref{eq:sipr}$. For a fixed $\epsilon$, and taking $\mu \rightarrow 0^+$, we obtain $\mathsf{val}\eqref{eq:sip} - \epsilon \leq \ell$. This being true for any $\epsilon > 0$, we deduce that $\mathsf{val}\eqref{eq:sip}\leq \ell$, which proves the equality.
\end{proof}
Setting the regularization parameter $\mu$ in practice implies a trade-off between the computational tractability of the semi-infinite program \eqref{eq:sipr} and the accuracy of the approximation of the original problem \eqref{eq:sip}. To solve the formulation \eqref{eq:sipr}, we propose to use a standard algorithm for convex semi-infinite programming, called cutting-plane (CP) algorithm \cite{cerulli_convergent_2022}. To that extent, we assume to have an \textit{separation oracle} computing, for any $\theta \in \mathbb{R}^N$,
\begin{align}
\phi(\theta) = \max_{(a,b) \in \mathcal{Y}} a^\top \theta + b,
\label{eq:oracle}
\end{align}
and an associate argmaximum. Solving the optimization problem in Eq.~\eqref{eq:oracle} may be computationally intensive, since the compact set $\mathcal{Y}$ may not be convex. Therefore, we only assume to have an oracle with relative optimality gap $\delta \in [0,1)$ computing $(a,b) \in \mathcal{Y}$, such that $\phi(\theta)- (a^\top \theta + b)  \leq \delta |\phi(\theta)|$. We treat this oracle as a black box, regardless of its implementation, via global optimization, gridding, interval arithmetics or sampling for instance.
\begin{algorithm}[H]
\caption{Cutting-plane algorithm for \eqref{eq:sipr}}
\label{chap:alg:CP}
\hspace*{\algorithmicindent}\textbf{{Input:}} {An oracle with parameter $\delta \in [0,1)$, a tolerance $\epsilon \in \mathbb{R}_+$, a finite set $\mathcal{Y}^0 \subset \mathcal{Y}$, $k\gets0$}
\begin{algorithmic}[1]
    \setcounter{ALG@line}{-1}
    \While{true}
    	\State Compute $\theta^k$, the solution of the convex Quadratic Programming problem
    	\begin{align}
\begin{array}{rll}
\underset{\theta \in \mathbb{R}^N}{\max} & c^\top \theta - \frac{\mu}{2} \lVert \theta \rVert^2 &  \\
\text{s.t.} & a^\top \theta + b \leq 0 & \forall (a,b) \in \mathcal{Y}^k.
\end{array} \label{eq:master}
\end{align}
    
	    \State Call the oracle to compute $(a^k,b^k)$ an approximate solution of \eqref{eq:oracle} with relative optimality gap $\delta$.
        \label{step:y^k}
	    \If {$(a^k)^\top \theta^k + b^k \leq \epsilon$} 
	        \State  Return $\theta^k$.    \label{stepterm}
	    \Else
	    	\State{$\mathcal{Y}^{k+1} \gets \mathcal{Y}^k \cup \{ (a^k,b^k)\} $}
	        \State {$k \gets k + 1$}
	    \EndIf
    \EndWhile
\end{algorithmic}
\end{algorithm}
Before stating the termination and the convergence of Algorithm~\ref{chap:alg:CP}, we introduce the vector $\hat{\theta} \in \mathbb{R}^N$ of coordinates of the polynomial $\hat{v}(t,x) = t - 1 - T$ in the basis $\Phi(t,x)$, and we notice that this element helps obtaining feasible solutions since $\phi(\hat{\theta}) = -1$: Due to Lemma~\ref{lem:feas}, we observe that if $\theta$ has a feasibility error less or equal than $\eta \geq 0$ in \eqref{eq:sip}  and \eqref{eq:sipr} then, $\theta + \eta \hat{\theta}$ is feasible in \eqref{eq:sip} and \eqref{eq:sipr}. For any $\mu > 0$, we define the convex and compact set $\mathcal{X}_\mu = \{ \theta \in \mathbb{R}^N \colon c^\top \theta - \frac{\mu}{2} \lVert \theta \rVert^2 \geq c^\top \hat{\theta} - \frac{\mu}{2} \lVert \hat{\theta} \rVert^2 \}$, and we define $R_\mu = \sup_{\theta \in \mathcal{X}_\mu } \lVert \theta \rVert$. Finally, we define the function $r_\mu(e) = e (1+T + \mu R_\mu^2 (1+\frac{e}{2}))$. Note that $r_\mu(e) \underset{e \rightarrow 0}{\rightarrow} 0$.
\begin{theorem}
If $\epsilon > 0$, Algorithm~\ref{chap:alg:CP} stops after a finite number $K$ of iterations, and $\theta^K + \frac{\epsilon}{1-\delta} \hat{\theta}$ is a feasible and  $r_\mu(\frac{\epsilon}{1-\delta})$-optimal in \eqref{eq:sipr}. If $\epsilon = 0$, the alternative holds: (a) Algorithm~\ref{chap:alg:CP} either stops after a finite number of iterations, and the last iterate is the optimal solution of \eqref{eq:sipr}, (b) Or it generates an infinite sequence, and the optimality gap and the feasibility error converge towards zero with an asymptotic rate in $O(\frac{1}{k})$. 
\end{theorem}
\begin{proof}
First of all, we notice that during the execution of Algorithm~\ref{chap:alg:CP}, we necessarily have $\theta^k \in \mathcal{X}_\mu$, since $\hat{\theta}$ is a feasible solution in \eqref{eq:master} with value $c^\top \hat{\theta} - \frac{\mu}{2} \lVert \hat{\theta} \rVert^2$, therefore by optimality of $\theta^k$ in \eqref{eq:master}, $c^\top \theta^k - \frac{\mu}{2} \lVert \theta^k \rVert^2 \geq c^\top \hat{\theta} - \frac{\mu}{2} \lVert \hat{\theta} \rVert^2$. The finite convergence of Algorithm~\ref{chap:alg:CP} if $\epsilon >0$, and the convergence rate in the case $\epsilon = 0$ (if no finite convergence) follows from \cite[Th.1.1-1.2]{oustry_global_2023} (which is an extension of \cite{cerulli_convergent_2022}) : we apply these theorems to the problem \eqref{eq:sipr} with the additional constraint $\theta \in \mathcal{X}_\mu$. As previously explained, this additional constraint does not change the execution of the algorithm, but it enables us to satisfy the compactness assumption of \cite[Th.1.1-1.2]{oustry_global_2023}. We also note that the objective function is $\mu$-strongly concave, and that $\hat{\theta} \in \mathcal{X}_\mu$ is a strictly feasible point with respect to the semi-infinite constraints.

We finish the proof by showing that if Algorithm~\ref{chap:alg:CP} stops at iteration $K$, then  $\tilde{\theta} = \theta^K + \frac{\epsilon}{1-\delta} \hat{\theta}$ is feasible and  $r_\mu(\frac{\epsilon}{1-\delta})$-optimal in \eqref{eq:sipr}.  If Algorithm~\ref{chap:alg:CP} stops at iteration $K$, this means that $(a^K)^\top \theta^K + b^K \leq \epsilon$. If $\phi(\theta^K)\leq 0$, then $\theta^K$ is feasible in \eqref{eq:sipr}, and so is $\tilde{\theta}$ due to Lemma~\ref{lem:feas}. If $\phi(\theta^K) > 0$, then by property of the $\delta$-oracle, 
$(1-\delta) \phi(\theta^K) \leq (a^K)^\top \theta^K + b^K \leq \epsilon$, and we deduce that the feasibility error is $\phi(\theta^K) \leq \frac{\epsilon}{1-\delta}$. With Lemma~\ref{lem:feas}, we deduce that $\tilde{\theta}$ is feasible in \eqref{eq:sipr}. We also note that 
\begin{align}
c^\top \tilde{\theta} - \frac{\mu}{2} \lVert \tilde{\theta} \rVert^2  &=  c^\top \theta^K + \frac{\epsilon}{1-\delta} c^\top \hat{\theta}  - \frac{\mu}{2} \lVert \theta^K + \frac{\epsilon}{1-\delta} \hat{\theta} \rVert^2 \\
& \geq c^\top \theta^K - \frac{\epsilon}{1-\delta} (1+T)  - \frac{\mu}{2} \left( \lVert \theta^K \rVert^2  + \frac{2 \epsilon}{1-\delta} \lVert \theta^K \rVert \: \lVert\hat{\theta} \rVert + \frac{\epsilon^2}{(1-\delta)^2} \lVert \hat{\theta} \rVert^2  \right),
\end{align}
since $c^\top \hat{\theta} = V_{\hat{\theta}}(0,x_0) = -(1+T)$, and due to the Cauchy-Schwartz inequality. By optimality of $\theta^K$ in \eqref{eq:master}, which is a relaxation of \eqref{eq:sipr}, we know that $\mathsf{val}\text{\eqref{eq:sipr}} \leq  c^\top \theta^K - \frac{\mu}{2} \lVert \theta^K \rVert^2$. Applying this, we deduce that
\begin{align}
c^\top \tilde{\theta} - \frac{\mu}{2} \lVert \tilde{\theta} \rVert^2  
& \geq \mathsf{val}\text{\eqref{eq:sipr}}- \frac{\epsilon}{1-\delta} (1+T)  - \frac{\mu}{2} \left( \frac{2 \epsilon}{1-\delta} \lVert \theta^K \rVert \: \lVert\hat{\theta} \rVert + \frac{\epsilon^2}{(1-\delta)^2} \lVert \hat{\theta} \rVert^2  \right) \\
& \geq \mathsf{val}\text{\eqref{eq:sipr}}- \frac{\epsilon}{1-\delta} (1+T)  - \mu R_\mu^2 \left( \frac{ \epsilon}{1-\delta}  + \frac{\epsilon^2}{2(1-\delta)^2}  \right) \\
& \geq \mathsf{val}\text{\eqref{eq:sipr}}- r_\mu(\frac{\epsilon}{1-\delta}),
\end{align}
the second inequality following from the fact that $\lVert \hat{\theta} \rVert \leq R_\mu$ and $\lVert \theta^K \rVert \leq R_\mu$, as $\hat{\theta}, \theta^K \in \mathcal{X}_\mu$.
\end{proof}

\section{Feedback control based on approximate value functions}
\label{sec:control}
In the previous section, we have seen how to compute subsolutions of the HJB equation based on convex semi-infinite programming, and how to deduce a lower bound on the minimal travel time. In this section, we focus on how subsolutions of the HJB equation, which approximate the value function $V^*$, enable one to recover a near-optimal control for the minimal time control problem \eqref{eq:system}-\eqref{eq:inf}. 

\subsection{Controller design and existence of trajectories}
For a given continuously differentiable function $V \in C^1(\mathbb{R}^{n+1})$, we define the set-valued maps 
\begin{align}
\mathcal{U}_V(t,x) &=  \underset{u \in U}{\mathsf{argmin}} \:   \nabla_x V (t,x)^\top  f(t,x,u) \label{eq:selection} \\
\mathcal{I}_V(t,x) &= \{ u \in \mathcal{U}_V(t,x) \colon f(t,x,u) \in T_X(x) \},
\end{align}
where $T_X(x)$ is the contingent cone to $X$ at point $x$ (see Introduction). In line with previous works designing feedback controllers based on approximate value functions \cite{henrion_nonlinear_2008,jones_polynomial_2023}, we are interested in the trajectories satisfying the following differential inclusion depending on the function $V \in C^1(\mathbb{R}^{n+1})$:
\begin{align}
\dot{x}_V(t) = f(t,x_V(t),u_V(t))  \text{ with } u_V(t) \in \mathcal{U}_V(t,x_V(t)). \label{eq:diffinclusion}
\tag{\mbox{$CL_V$}}
\end{align}
Intuitively, such a feedback control pushes the system towards the descent direction of the function $V$. The following proposition confirms that, should the function $V \in C^1(\mathbb{R}^{n+1})$ be optimal in problem \eqref{eq:lpc1}, then any minimal time trajectory satisfies the differential inclusion  \eqref{eq:diffinclusion} with respect to $V$. 
\begin{proposition} Under Assumptions~\ref{as:convexF}-\ref{as:finite}, we consider an optimal trajectory $(x^*(\cdot),u^*(\cdot))$ of the minimal time control problem \eqref{eq:system}-\eqref{eq:inf} starting from $(0,x_0)$, with hitting time $\tau^* = V^*(0,x_0)$. If the linear program \eqref{eq:lpc1}, for $\mathcal{F} = C^1(\mathbb{R}^{n+1})$, admits an optimal solution $V$, then, for almost every $t \in [0, \tau^*],$
\begin{align}
u^*(t) \in \mathcal{I}_V(t,x^*(t)) \subset \mathcal{U}_{V}(t,x^*(t)). 
\label{eq:optcommand}
\end{align}
In particular, the trajectory $(x^*(\cdot),u^*(\cdot))$ satisfies the differential inclusion \eqref{eq:diffinclusion}.
\label{prop:opttraj}
\end{proposition}
\begin{proof} We define the function $\alpha(t) = V (t,x^*(t)) + t$, which is differentiable. We have that $\alpha'(t) = \partial_t V (t,x^*(t)) + 1 + \nabla_x V (t,x^*(t))^\top  f(t,x^*(t),u^*(t))$, for almost all $t \in [0,\tau^*]$. Since $V$ is feasible in \eqref{eq:lpc1}, therefore  satisfies Eq.~\eqref{eq:subsol1}, and since $(x^*(t),u^*(t)) \in X \times U$ a. e. on $[0,\tau^*]$, we know that $\alpha'(t) \geq 0$ a. e. on $[0,\tau^*]$. This proves that the differentiable function $\alpha(t)$ is non-decreasing function over $[0,\tau^*]$. By optimality of $V$ in \eqref{eq:lpc1}, and due to Th.~\ref{th:duality} (Assumptions~\ref{as:convexF}-\ref{as:finite} are satisfied), $\alpha(0) = V(0,x_0) = \mathsf{val}\eqref{eq:lpc1} = \tau^*$. Moreover,  $\alpha(\tau^*)= \tau^* +V(\tau^*,x^*(\tau^*)) \leq  \tau^*$, since $V$ satisfies Eq.~\eqref{eq:subsol2} and $x^*(\tau^*) \in K$. From $\alpha(\tau^*) \leq \alpha(0)$, we obtain that $\alpha(t)$ is constant. Hence, $\partial_t V(t,x^*(t)) + 1 + \nabla_x V (t,x^*(t))^\top  f(t,x^*(t),u^*(t)) = 0$, meaning 
\begin{align}
\nabla_x V (t,x^*(t))^\top  f(t,x^*(t),u^*(t)) = -(\partial_t V (t,x^*(t)) + 1),  \text{ a.e. on } [0,\tau^*]. \label{eq:val}
\end{align}
As  $V$ satisfies Eq.~\eqref{eq:subsol1}, we have that $\nabla_x V (t,x^*(t))^\top  f(t,x^*(t),u) \geq -(\partial_t V (t,x^*(t)) + 1)$ for all $t \in [0,\tau^*]$ and for all $u\in U$. Together with Eq.~\eqref{eq:val}, we deduce that $u^*(t) \in \mathcal{U}_V(t,x^*(t))$ for almost all $t \in [0, \tau^*]$. Based on this fact, Lemma~\ref{lem:tangentcone} yields that for almost all $t \in [0, \tau^*]$, $f(t,x^*(t),u^*(t)) \in T_X(x^*(t))$. Therefore, for almost all $t \in [0, \tau^*]$, $u^*(t) \in \mathcal{I}_V(t,x^*(t))$.
\end{proof}
We just saw that whenever $V \in C^1(\mathbb{R}^{n+1})$ is optimal in the linear program \eqref{eq:lpc1}, any minimal time trajectory is a solution of the differential inclusion \eqref{eq:diffinclusion} associated with the function $V$. However, we may not be able to compute exactly such an optimal function in practice, especially because it may not exist. The next theorem states the existence of closed-loop trajectories following \eqref{eq:diffinclusion}, for any function $V \in C^1(\mathbb{R}^{n+1})$.
\begin{theorem}
Under Assumptions~\ref{as:convexF}-\ref{as:finite}, if  $V \in C^1(\mathbb{R}^{n+1})$ is such that for any $(t,x) \in \mathbb{R}_+ \times X, \mathcal{I}_V(t,x) \neq \emptyset$, then there exists a trajectory $(x_V(\cdot),u_V(\cdot))$ starting at $(0,x_0)$, satisfying the differential inclusion \eqref{eq:diffinclusion} over $[0,\infty)$ and such that $x_V(t) \in X$ for almost all $t \in [0,\infty)$.
\label{th:existencetraj}
\end{theorem}
\begin{proof}
We introduce an auxiliary control system to reduce to a time-invariant system with a convex control set, so as to fit in the setting of \cite[Th.~6.6.6]{aubin_jean-pierre_viability_1991}. In what follows, we use the notation $y = (t,x)$ again. We introduce two objects: the set-valued  map $\hat{U}(y) = \{ f(y,u), u \in U \}$ and the function $\hat{f}(y,v) = \begin{pmatrix}
1 \\ v
\end{pmatrix}$ for $y \in \mathbb{R}^{n+1}$ and $v \in  \mathbb{R}^{n}$. According to the terminology introduced in \cite[Def.~6.1.3]{aubin_jean-pierre_viability_1991}, $(\hat{U},\hat{f})$ is a Marchaud control system, as (i) $\{(y,v) \in \mathbb{R}^{2n+1} \colon v \in \hat{U}(y)\} $ is closed, (ii) $\hat{f}$ is continuous, (iii) the velocity set $\{1 \} \times f(y,U)$ is convex according to Assumption~\ref{as:convexF} and (iv) $\hat{f}$ has a linear growth, and so has $\hat{U}$ due to the fact that $f$ is Lipschitz continuous and $\mathcal{U}$ is bounded. We introduce $\mathcal{C} = \mathbb{R}_+ \times X$ and define the regulation map $REG(y) = \{ v \in \hat{U}(y) \colon \hat{f}(y,v) \in  T_{\mathcal{C}}(y) \}$. We also introduce the set-valued map  $SEL(y) = \underset{v \in \hat{U}(y)}{\text{argmin}} \: \nabla_x V(y)^\top v$. We prove now that the graph of $SEL$ is closed. For any converging sequence $(y_k,v_k) \rightarrow (\bar{y},\bar{v})$ with $v_k \in SEL(y_k)$, we see that for all $k \in \mathbb{N}$, $v_k = f(y_k,u_k)$ for a given $u_k \in U$ and $\nabla_x V(y_k)^\top f(y_k,u_k) = h(y_k)$, where $h(y_k) = \min_{u \in U} \nabla_x V(y_k)^\top f(y_k,u)$. Up to extracting a subsequence of $u_k$, we can assume that $u_k \rightarrow \bar{u}$, as $U$ is compact. Note that $h$ is continuous, by application of the Maximum Theorem \cite[Th.~2.1.6]{aubin_jean-pierre_viability_1991}, in so far as (i) $(y,u) \mapsto \nabla_x V(y)^\top f(y,u)$ is continuous, therefore lower and upper semicontinuous, (ii) the set-valued map $M(y)= U$ is compact-valued, and lower and upper semicontinuous since it is constant.   By continuity of $h$, $\nabla V$ and $f$, we conclude that $\nabla_x V(\bar{y})^\top \bar{v} = \nabla_x V(\bar{y})^\top f(\bar{y},\bar{u}) = h(\bar{y}) = \underset{v \in \hat{U}(\bar{y})}{\min} \: \nabla_x V(\bar{y})^\top v$, meaning that $\bar{v} = f(\bar{y},\bar{u}) \in SEL(\bar{y})$.

We notice that if $u \in \mathcal{I}_V(y)$, then $v = f(y,u) \in REG(y) \cap SEL(y)$. As $\mathcal{I}_V(y) \neq \emptyset$ for all $y \in \mathcal{C}$ (by assumption), $REG(y) \cap SEL(y) \neq \emptyset$. Together with the closedness of the graph of $SEL$, this means $SEL$ is a selection procedure of $REG$, according to the terminology of \cite[Def.~6.5.2]{aubin_jean-pierre_viability_1991}, and has convex values. We underline that  $REG(y) \neq \emptyset$, for all $y \in \mathcal{C}$, \textit{i.e.}, $\mathcal{C}$ is a viability domain for $(\hat{U},\hat{f})$.  As $(0, x_0) \in \mathcal{C}$, \cite[Th.~6.6.6]{aubin_jean-pierre_viability_1991} yields the existence of a solution $(y(\cdot), v(\cdot))$ such that $y(t) \in \mathcal{C}$, $v(t) \in REG(y(t))$ and 
\begin{align}
v(t) \in SEL(y(t)) \cap REG(y(t)),
\label{eq:argmin}
\end{align}
 for almost all $t \in [0, \infty)$. We notice first that $y_1(0) = 0$ and $\dot{y}_1(t) = 1$ for almost all $t \geq 0$, thus $y_1(t) = t$. Hence, we can indeed see $y(t)$ as $(t,x(t))$, with  $x(0) = x_0$ and $\dot{x}(t) = v(t)$. Moreover, $v(t) = f(t,x(t),u(t))$ for a given $u(t) \in U$, since $v(t) \in \hat{U}(y(t)) = f(t,x(t),U)$ a.e. on $[0,\infty)$. We deduce from $v(t) \in SEL(y(t))$, which comes from \eqref{eq:argmin}, that $u(t) \in \mathcal{U}_V(t,x(t))$. Moreover, we deduce from $y(t) \in \mathcal{C}$ that $x(t) \in X$ a.e. on $[0,\infty)$.
\end{proof}

\begin{remark}
The condition $\mathcal{I}_V(t,x) \neq \emptyset$ in Th.~\ref{th:existencetraj} may appear restrictive, because it is not evident why a vector $f(t,x,u_V)$ minimizing $\nabla_x V (t,x)^\top  f(t,x,u)$ over $u \in U$ would belong to $T_X(x)$. However, we have seen that under the hypotheses of Prop.~\ref{prop:opttraj}, Eq.~\eqref{eq:optcommand} yields $\mathcal{I}_V(t,x) \neq \emptyset$. Moreover, should the condition $\mathcal{I}_V(t,x) \neq \emptyset$ not be satisfied, we could enlarge the definition of $\mathcal{U}_V(t,x)$ in $\mathcal{U}_{V,\epsilon}(t,x) = \underset{u \in U}{\mathsf{argmin}_\epsilon} \:   \nabla_x V (t,x)^\top  f(t,x,u)$, so that for $\epsilon > 0$ large enough,  $\mathcal{I}_{V, \epsilon}(t,x) = \{ u \in \mathcal{U}_{V,\epsilon}(t,x) \colon f(t,x,u) \in T_X(x) \}\neq \emptyset$.
\end{remark}

\subsection{Performance of the feedback controller depending on the value function approximation error}
Previously, we introduced closed-loop trajectories satisfying the differential inclusion \eqref{eq:diffinclusion} with respect to a function $V \in C^1(\mathbb{R}^{n+1})$. In this section, we state some performance guarantees on those trajectories, depending on some properties of the function $V$. In the following, we assume that, up to an enlargement of the time horizon, the system can reach the target set starting from any initial condition $(t,x) \in [0, T] \times X$, and that the associated value function is Lipschitz.
\begin{assumption}
There exists a time $T^\sharp \geq T$ such that the minimal time control problem \eqref{eq:system}-\eqref{eq:inf} defined over $[0, T^\sharp]$ has a value function $V^\sharp$ which takes finite values over $Y = [0,T] \times X$, and is Lipschitz continuous.
\label{as:regV}
\end{assumption}
We emphasize that, under Assumption~\ref{as:regV},  $V^*(t,x) < \infty$ implies $V^*(t,x) = V^\sharp(t,x)$ for any $(t,x) \in Y$. Since $V^\sharp$ is Lipschitz continuous over $Y \subset \mathbb{R}^{n+1}$, it admits a Lipschitz continuous extension over $\mathbb{R}^{n+1}$ \cite[Chap.~3, Th.~1]{gariepy_measure_2015}. We assimilate the value function and its extension on $\mathbb{R}^{n+1}$, such that we can speak about the Clarke's generalized derivative $\partial^c V^\sharp(y)$ of $V^\sharp$ at $y \in Y$. For any $V\in C^1(\mathbb{R}^{n+1})$, we introduce the notation 
\begin{align}
\lVert \nabla V - \nabla V^\sharp \rVert_\infty = \sup_{y \in Y } \sup_{g\in\partial^c V^\sharp(y)} \lVert \nabla V(y) - g \rVert_2.
\end{align}
We also define the constant $C_f = \sup_{(t,x,u) \in Y \times U} \lVert f(t,x,u) \rVert <\infty$. %
\begin{theorem}
Let $V\in C^1(\mathbb{R}^{n+1})$ be a continuously differentiable function, and let $(x_V(\cdot),u_V(\cdot))$ be a closed-loop trajectory starting at $(0,x_0)$ satisfying the differential inclusion \eqref{eq:diffinclusion} and the state constraints over $[0, T]$.  We define $t_V = \sup \{t \in [0,T] \colon x_V([0,t])\subset  X \setminus K \}$. Then, under Assumptions~\ref{as:convexF}-\ref{as:regV},
\begin{align}
V^\sharp(t, x_V(t)) \leq (\tau^* - t)  + t \: 2 (1 + C_f) \lVert \nabla V - \nabla V^\sharp \rVert_\infty \quad \forall t \in[0, t_V],
\label{eq:timeerror}
\end{align}
where $\tau^* = V^*(0,x_0) =  V^\sharp(0,x_0) \leq t_V$. In particular, we notice that
\begin{align}
V^\sharp(\tau^*, x_V(\tau^*)) \leq 2 \tau^*  (1 + C_f) \lVert \nabla V - \nabla V^\sharp \rVert_\infty. \label{eq:timeerrortau}
\end{align}
\label{th:perf}
\end{theorem} In Eq.~\eqref{eq:timeerrortau},  $V^\sharp(\tau^*, x_V(\tau^*))$ measures how far the closed-loop trajectory $(x_V(\cdot),u_V(\cdot))$ is from the target set $K$ at the moment when the time-optimal trajectory reaches $K$. As a corollary, we give a condition for the closed-loop trajectory $(x_V(\cdot),u_V(\cdot))$ to effectively reach the target set $K$, with a bounded delay compared to the time-optimal trajectory.
\begin{corollary}
Under the same hypotheses as Th.~\ref{th:perf}, if  $\lVert \nabla V - \nabla V^\sharp \rVert_\infty \leq \frac{1 - \tau^*/T}{2 (1 + C_f)}$, then 
\begin{align}
x_V(t_V) \in K \text{ with } t_V \in [\tau^*, \frac{1}{1 - 2 (1 + C_f) \lVert \nabla V - \nabla V^\sharp\rVert_\infty}\tau^*].
\end{align}
\label{cor:perf}
\end{corollary}
We underline that the hitting time $t_V \geq \tau^*$ converges to the minimal time $\tau^*$, when the approximation error $\lVert \nabla V - \nabla V^\sharp\rVert_\infty$ vanishes.
\begin{proof}[Proof of Th~\ref{th:perf} and Cor.~\ref{cor:perf}] 
For any Lipschitz continuous function $F: \mathbb{R}^{n+1} \to \mathbb{R}$, we recall that $\partial^c F(y)$ denote the Clarke's generalized derivative at $y$, and we define $H_F$ as 
\begin{align} \label{eq:defHmin}
H_F(y) =  1 + \min_{\begin{subarray}{c} u \in U \\ g \in \partial^c F(y) \end{subarray}} \{ g^\top \begin{pmatrix}
1 \\ f(y,u)
\end{pmatrix} \}.
\end{align}
The minimum is attained by continuity of the objective, and by compactness of $U$ and $\partial^c F(y)$ (see \cite{clarke_generalized_1975}). Note also that for any $V  \in C^1(\mathbb{R}^{n+1})$, for any $y = (t,x) \in \mathbb{R}^{n+1}$, $H_V(t,x) = 1 + \partial_t V(t,x) + \min_{u \in U} \nabla_x V(t,x)^\top f(t,x,u)$, and the argmin is $\mathcal{U}_V(t,x)$. By application of the Maximum Theorem \cite[Th.~2.1.6]{aubin_jean-pierre_viability_1991},  we know that $H_{F}$ is lower semi-continuous, since (i) $\partial^c F(y)$ is a compact-valued and upper semi-continuous set-valued map \cite{clarke_generalized_1975}, therefore so is $y \mapsto U \times \partial^c F(y)$, and (ii) $(y,u,g) \mapsto g^\top \begin{pmatrix}
1 \\ f(y,u)
\end{pmatrix} $ is continuous.

First, we take any $y_1 = (t_1,x_1) \in [0, T] \times X \setminus K$, and we prove that $H_{V^\sharp}(t_1,x_1) \leq 0$. According to Assumption~\ref{as:regV}, $V^\sharp(t_1,x_1) < \infty$, and according to Th.~\ref{th:exopti} applied to the control system \eqref{eq:system}-\eqref{eq:inf} on the interval $[0 , T^\sharp]$, there exists an optimal trajectory $(x(\cdot), u(\cdot))$ over $[t_1, t_2]$ (with $t_2 > t_1$ since $x_0 \notin K$) starting from $(t_1, x_1)$. By definition, $V^\sharp(t_1,x_1) = t_2 - t_1$. We can also prove that for all $t \in [t_1, t_2]$, $V^\sharp(t,x(t)) = t_2 - t$: (i) the trajectory restricted to $[t, t_2]$, yields an admissible trajectory starting from $(t,x(t))$, therefore $V^\sharp(t,x(t)) \leq t_2 - t$, and (ii) for an optimal trajectory $(\tilde{x}(\cdot), \tilde{u}(\cdot))$ starting from $(t,x(t))$ over $[t, t_3]$, the trajectory following $(x(\cdot), u(\cdot))$ over $[t_1, t]$ and $(\tilde{x}(\cdot), \tilde{u}(\cdot))$ over $[t, t_3]$ is admissible and starting from $(t_1, x_1)$, therefore, $V^\sharp(t,x(t)) + (t-t_1) \geq V^\sharp(t_1,x_1) = t_2 - t_1$, giving $V^\sharp(t,x(t)) \geq t_2 - t$. As $\alpha(t) = V^\sharp(t,x(t)) = t_2 - t$ for all $t \in [t_1, t_2]$, we deduce that  
\begin{align}
\alpha'(t) = -1 \text{ a. e. on } [t_1, t_2].
\label{eq:constantderivative}
\end{align}
Moreover, since $V^\sharp$ is Lipschitz continuous by assumption, and $t \mapsto (t, x(t))$ is Lipschitz continuous as $x(t)$ is differentiable a.e. with a bounded derivative, Lemma~\ref{lem:clarkedifferential} gives:
$\alpha'(t) = \frac{d (V^\sharp(t,x(t)))}{dt} \geq \min_{g \in \partial^c V^{\sharp}(t,x(t))} g^\top \begin{pmatrix}
1 \\ \dot{x}(t)
\end{pmatrix}$ a.e. on $[t_1, t_2]$. Using that $\dot{x}(t) = f(t,x(t),u(t))$ a.e. on $[t_1, t_2]$, and the definition of $H_{V^\sharp}$:

 \begin{align}
\alpha'(t) \geq \min_{g \in \partial^c V^{\sharp}(t,x(t))} g^\top \begin{pmatrix}
1 \\ f(t,x(t),u(t))
\end{pmatrix}
 \geq H_{V^\sharp}(t,x(t)) - 1,
\end{align}
a.e. on $[t_1, t_2]$. Combining this with Eq.~\eqref{eq:constantderivative}, we deduce that for almost all $t \in [t_1, t_2]$, $H_{V^\sharp}(t,x(t)) \leq 0$. By lower semi-continuity of $H_{V^{\sharp}}$ (see above), and by continuity of $x(\cdot)$
\begin{align}
H_{V^{\sharp}}(t_1,x_1) \leq 0. \label{eq:negeps}
\end{align}

Second, still for any $(t_1,x_1) \in [0, T] \times X \setminus K$, we observe that there exists $(g,u_1) \in \partial^c V^\sharp(t_1,x_1) \times U$ such that $H_{V^\sharp}(t_1,x_1) = 1 + g^\top \begin{pmatrix}
1 \\ f(t_1,x_1,u_1)
\end{pmatrix}$; indeed, we already mentioned that the minimum in \eqref{eq:defHmin} is attained. Therefore, for any $V \in C^1(\mathbb{R}^{n+1})$
\begin{align}1 + \partial_t V(t_1,x_1) + \nabla_x V (t_1,x_1)^\top f(t_1,x_1,u_1) & = H_{V^{\sharp}}(t_1,x_1) +   (\nabla V(t_1,x_1) - g)^\top \begin{pmatrix}
1 \\ f(t_1,x_1,u_1)\end{pmatrix}  \\ & \leq H_{V^{\sharp}}(t_1,x_1) +  \lVert \nabla V - \nabla V^\sharp \rVert (1+C_f), \label{eq:chasles}
\end{align}
the inequality being due to  Cauchy-Schwartz inequality, and the definition of $\lVert \nabla V - \nabla V^\sharp \rVert$. We know that  $H_{V}(t_1,x_1) \leq \partial_t V(t_1,x_1) + 1 + \nabla_x V (t_1,x_1)^\top f(t_1,x_1,u_1)$ by definition of $H_{V}(t_1,x_1)$ (as $u_1 \in U$),  therefore Eq.~\eqref{eq:chasles} gives $H_{V}(t_1,x_1) \leq H_{V^{\sharp}}(t_1,x_1) + (1+C_f) \lVert \nabla V - \nabla V^\sharp \rVert$. Using this inequality and Eq.~\eqref{eq:negeps}, we deduce that for all $(t_1,x_1) \in [0, T] \times X \setminus K$,
\begin{align}
H_{V}(t_1,x_1) \leq (1+C_f) \lVert \nabla V - \nabla V^\sharp \rVert.
\label{eq:boundeps}
\end{align}
Third, according to the hypotheses of the theorem, we take any $V\in C^1(\mathbb{R}^{n+1})$, and any closed-loop trajectory $(x_V(\cdot),u_V(\cdot))$ starting at $(0,x_0)$ satisfying the differential inclusion \eqref{eq:diffinclusion} and the state constraints over $[0, T]$. We, then, study the evolution of $V^\sharp$ over this trajectory. As  $x_V(t)$ is Lipschitz continuous, Lemma~\ref{lem:clarkedifferential} yields the existence of $g(t) \in \partial^c V^\sharp(t,x_V(t))$ for almost all $t \in [0, T]$, such that 
\begin{align}
\frac{d}{dt}\left(V^\sharp(t,x_V(t))\right) & \leq g(t)^\top \begin{pmatrix}
1 \\ f(t,x_V(t),u_V(t))
\end{pmatrix} \text{ a.e. on } [0, T].
\end{align}
As $u_V(t) \in \mathcal{U}_V(t,x_V(t))$, we know that $H_V(t,x_V(t)) = 1 + \nabla V(t,x_V(t))^\top \begin{pmatrix}
1 \\ f(t,x_V(t),u_V(t)),
\end{pmatrix}$, and therefore,
\begin{align}
   \frac{d}{dt}\left(V^\sharp(t,x_V(t))\right) & \leq  -1 + H_V(t,x_V(t)) + (g(t) - \nabla V(t,x_V(t)))^\top \begin{pmatrix}
1 \\ f(t,x_V(t),u_V(t)),
\end{pmatrix}
\end{align}
We deduce, using Cauchy-Schwartz inequality and the definition of $\lVert \nabla V - \nabla V^\sharp \rVert$,
\begin{align}
\frac{d}{dt}\left(V^\sharp(t,x_V(t))\right) \leq -1 + H_V(t,x_V(t)) +  (1+C_f) \lVert \nabla V - \nabla V^\sharp \rVert, \label{eq:evvcirc}
\end{align}
for almost all $[0, T]$. Moreover, for all $t \in [0, t_V)$, $x_V(t) \notin K$. Therefore, we can apply Eq.~\eqref{eq:boundeps} to deduce, in combination with Eq.~\eqref{eq:evvcirc}, that for almost all $[0, t_V]$,$
\frac{d}{dt}\left(V^\sharp(t,x_V(t))\right) \leq -1 + 2  (1+C_f) \lVert \nabla V - \nabla V^\sharp \rVert$. By integration, we deduce that for all $t \in [0, t_V]$, $V^\sharp(t, x_V(t)) - \tau^* \leq - t + 2 t (1+C_f) \lVert \nabla V - \nabla V^\sharp \rVert$, as $V^\sharp(0,x_V(0)) = V^\sharp(0,x_0) = \tau^*$. This proves Eq.~\eqref{eq:timeerror}. 

Fourth and finally, we prove the corollary. Due to the definition of $t_V$, the following (non-exclusive) alternative holds: either $x_V(t_V) \in K$ or $t_V = T$. Moreover, if  $\lVert \nabla V - \nabla V^\sharp \rVert_\infty \leq \frac{1 - \tau^*/T}{2 (1 + C_f)}$, then  $V^\sharp(t, x_V(t)) - \tau^* \leq - t +  t (1 - \tau^*/T)$ and $V^\sharp(t, x_V(t)) \leq \tau^* (1 - t/T)$ for all $t \in [0, t_V]$.  We notice that if $t_V= T$, then $V^\sharp(t_V, x_V(t_V)) \leq 0$, \textit{i.e.}, $x_V(t_V) \in K$. Coming to the aforementioned alternative, we deduce that $x_V(t_V) \in K$. Moreover, this fact combined with Eq.~\eqref{eq:timeerror} gives us that $0 \leq (\tau^* - t_V) + t_V \: 2 (1 + C_f) \lVert \nabla V - \nabla V^\sharp \rVert_\infty$, hence 
\begin{align}
t_V \left(1 - 2 (1 + C_f) \lVert \nabla V - \nabla V^\sharp \rVert_\infty \right)  \leq  \tau^*.
\label{eq:ineqtv}
\end{align} 
By assumption, $1 - 2 (1 + C_f) \lVert \nabla V - \nabla V^\sharp \rVert_\infty \geq \tau^*/T >0$, we can thus divide Eq.~\eqref{eq:ineqtv} by this quantity to obtain the result of the corollary: $t_V   \leq  \tau^*/\left( 1 - 2 (1 + C_f) \lVert \nabla V - \nabla V^\sharp \rVert_\infty \right)$.
\end{proof}
In the previous theorem and the corollary, we saw that the suboptimality, in terms of hitting time, of a closed-loop trajectory $(x_V(\cdot),u_V(\cdot))$ satisfying the differential inclusion \eqref{eq:diffinclusion} decreases as the approximation error $\lVert \nabla V - \nabla V^\sharp \rVert_\infty$ decreases. Furthermore, we see that the closed-loop trajectory comes close to optimality when the approximation error vanishes. We now study a sufficient condition under which the approximation $\lVert \nabla V_d - \nabla V^\sharp \rVert_\infty$ can be made arbitrarily small, using a polynomial $V_d(t,x)$ of sufficiently large degree $d \in \mathbb{N}$.

\subsection{A sufficient regularity condition for the existence of near-optimal controllers based on polynomials}
In the case where the value function is twice differentiable, there exist polynomials $V_d$ with such a vanishing approximation error $\lVert \nabla V_d - \nabla V^\sharp \rVert_\infty$, and that are near optimal solutions in the hierarchy of semi-infinite programs $\eqref{eq:lprd}$.
\begin{theorem} \label{th:convergenceC2hjb}
Under Assumptions~\ref{as:convexF}-\ref{as:regV}, if the value function $V^\sharp$ belongs to $C^2(\mathbb{R}^p | Y)$, and is a subsolution to the HJB equation, then there exist a sequence of polynomials $(V_d(t,x))_{d\in\mathbb{N}^*}$, with  $V_d(t,x) \in \mathbb{R}_d[t,x_1,\dots, x_n]$, and two constants $c_1, c_2 >0$, such that for all $d \in \mathbb{N}^*$,
\begin{itemize}
\item The polynomial $V_d(t,x)$ is feasible, and $\frac{c_1}{d}$-optimal in the problems \eqref{eq:lpc1} and \eqref{eq:lprd},
\item The following inequality holds: $\lVert \nabla V_d - \nabla V^\sharp \rVert_\infty  \leq \frac{c_2}{d}$.
\end{itemize}
\end{theorem}
Under these hypotheses, the polynomials $V_d(t,x)$ are subsolutions to the HJB equation, and form a maximizing sequence of the problem \eqref{eq:lpc1}; we also  notice that the hierarchy of semi-infinite programs \eqref{eq:lprd} converges in $O(\frac{1}{d})$ in terms of objective value. Moreover, according to Cor.~\ref{cor:perf}, for any sequence of closed-loop trajectories $(x_{V_d}(\cdot), u_{V_d}(\cdot))$, the associated hitting times converge to the minimal time $\tau^*$: this is a minimizing sequence of trajectories for the optimal time control problem \eqref{eq:system}-\eqref{eq:inf}. 
\begin{proof} By definition of $C^2(\mathbb{R}^{n+1} |Y)$, there exists a function $Q \in C^2(\mathbb{R}^{n+1})$ such that $V^\sharp(y) = Q(y)$ and $\nabla V^\sharp(y) = \nabla Q (y)$ for all $y \in Y$. In application of Lemma~\ref{lem:polapprox}, as $Q$ has a locally Lipschitz gradient since it is twice differentiable, there exists a constant $A > 0$, and a sequence of polynomials $(w_d(t,x))_{d \in \mathbb{N}^*}$ with $w_d(t,x) \in \mathbb{R}_d[t,x_1,\dots,x_n] $ and such that for all $(t,x) \in Y$, $| w_d(t,x) - Q(t,x) | \leq \frac{A}{d}$ and $\lVert \nabla w_d(t,x) - \nabla Q(t,x) \rVert_2 \leq \frac{A}{d}$. With $\alpha_d = \frac{A(1+C_f)}{d}$, and $\beta_d = \frac{A}{d}(1 + T + T C_f)$, we define the polynomial $V_d(t,x) = w_{d}(t,x) + \alpha_d t - \beta_{d} \in \mathbb{R}_d[t,x_1,\dots,x_n]$. First, we notice that $\lVert \nabla V_d - \nabla V^\sharp \rVert_\infty \leq \lVert \nabla w_{d}  - \nabla V^\sharp \rVert_\infty + \alpha_d \leq \frac{A ( 2 + C_f)}{d}$ for all $d\geq 1$. This proves the second point of the theorem, having defined the constant $c_2 = A(2+C_f)$, which is independent from $d$. We prove now the first point.  For all $d \geq 1$, and $(t,x,u) \in [0, T] \times X \times U$,
\begin{align}
\partial_t V_d(t,x) + 1 + \nabla_x V_d(t,x)^\top f(t,x,u) &= \alpha_d + \partial_t V^\sharp(t,x)  + 1 +  \nabla_x V^\sharp(t,x)^\top f(t,x,u) \\ 
& & \hspace{-4cm} + (\nabla w_{d}(t,x) - \nabla V^\sharp(t,x))^\top \begin{pmatrix}
1 \\ f(t,x,u)
\end{pmatrix}  \\
& \geq \alpha_d +  (\nabla w_{d}(t,x) - \nabla V^\sharp(t,x))^\top \begin{pmatrix}
1 \\ f(t,x,u)
\end{pmatrix},
\end{align}
as $V^\sharp$ is a subsolution to the HJB equation, hence satisfies Eq.~\eqref{eq:subsol1}. Using the Cauchy-Schwartz inequality, we obtain $\partial_t V_d(t,x) + 1 + \nabla_x V_d(t,x)^\top f(t,x,u) \geq \alpha_d - \lVert \nabla w_d(y) - \nabla V^\sharp(y) \rVert_2 (1 + C_f) \geq \alpha_d - \frac{A}{d} (1 + C_f) = 0$. This proves that $V_d$ satisfies Eq.~\eqref{eq:subsol1}.  It also satisfies Eq.~\eqref{eq:subsol2}, because for any $(t,x) \in [0, T] \times K$,
\begin{align}
V_d(t, x) &= w_{d}(t,x) + \alpha_d t  - \beta_d \\
& \leq V^\sharp(t,x) + \frac{A}{d} + \alpha_d t  - \beta_d \\
& \leq V^\sharp(t,x) + \frac{A}{d} + \alpha_d T  - \beta_d \\
& \leq V^\sharp(t,x) = 0.
\end{align}
since $\frac{A}{d} + \alpha_d T  - \beta_d = 0$ by definition of $\beta_d$, and since $x \in K$. We deduce that $V_d$ is feasible in \eqref{eq:lprd}. Its objective value is $V_d(0,x_0) \geq w_{d}(0, x_0) - \beta_d \geq V^\sharp(0, x_0) - \frac{A}{d} - \beta_d = V^\sharp(0, x_0) - \frac{c_1}{d}$, where $c_1 = A(2 + T + T C_f)$. As $V^\sharp(0, x_0) = V^*(0,x_0)$ due to Assumption~\ref{as:finite}, $V_d(0,x_0) \geq V^*(0,x_0) - \frac{c_1}{d} \geq \mathsf{val}\eqref{eq:lpc1} - \frac{c_1}{d} \geq \mathsf{val} \eqref{eq:lprd} - \frac{c_1}{d}$, and we therefore conclude that $V_d$ is $\frac{c_1}{d}$-optimal in \eqref{eq:lpc1} and $\eqref{eq:lprd}$.
\end{proof}
\begin{remark}
     Admittedly, the hypothesis in Th.~\ref{th:convergenceC2hjb} that the value function $V^\sharp$ belongs to $C^2(\mathbb{R}^p | \Sigma)$ is stringent. It is worth noting, however, that there exist systems that satisfy this hypothesis. Here is an example: $\dot{x}(t) = u(t)$, $x(t) \in X = [0,1]^2$, $\lVert u(t) \rVert \leq 1 $ and $K = \{0 \} \times [0,1]$. The value function associated with the horizon $T = \infty$ is $V^\sharp(t,x) = x_1$.
\end{remark}
\section{Illustrative examples}
\label{sec:numerics}
We implemented and tested the proposed methodology on three Minimal Time Control Problems: a generalization of the Zermelo problem, a regatta problem and a generalization of the Brockett integrator. The numerical examples in this section were processed with our \verb+Julia+ package \verb+MinTimeControl.jl+\footnote{This package is available at \url{github.com/aoustry/MinTimeControl.jl}}. In this implementation of Algorithm~\ref{chap:alg:CP}, the master problem \eqref{eq:master} is solved with the simplex algorithm of the commercial solver \verb+Gurobi 10.0+ \cite{gurobi_optimization_llc_gurobi_2021}. At each iteration, we add a maximum of 100 points to the set $\mathcal{Y}^k$. The separation oracle \eqref{eq:oracle} is implemented with a random sampling scheme (with 500,000 samples at each iteration to detect violated constraints), and with the global optimization solver \verb+SCIP 8+ \cite{bestuzheva_scip_2021}, for the certification at the last iterate. This solver is used with a relative tolerance $\delta = 10^{-4}$, and with a time limit of $10,000s$. We also precise that we compute a heuristic trajectory based on the particularities of each problem; this heuristic is not optimal, but provides an upper bound $T$ on the minimum time, and therefore, a relevant time horizon $[0,T]$. The trajectory resulting from the heuristic is used to initialize the set $\mathcal{Y}^0$, in the sense that we enforce the HJB inequality for some points of this trajectory. During the iterations of the algorithm, we obtain functions $V_{\theta^k}(t,x)$ and we simulate the associate feedback trajectory defined by the differential inclusion \eqref{eq:diffinclusion}; if the obtained trajectory reaches the target set, it gives us an upper bound. Those trajectories are also used to enrich the set $\mathcal{Y}^k$. For all the numerical experiments, the regularization parameter is $\mu = 10^{-5}$, and we use the tolerance $\epsilon = 10^{-3}$.

Table~\ref{tab:zermelo}, Table~\ref{tab:regatta} and Table~\ref{tab:brockett} present the numerical results for three different applications. The different columns of these tables are the following:
\begin{itemize}
\item ``$d \in \mathbb{N}$'' is the degree of the polynomial basis used.
\item ``Estimated value of \eqref{eq:lprd}'' stands for the value $V_\theta(0,x_0)$, where $\theta$ is the output of Algorithm~\ref{chap:alg:CP}, using the sampling oracle. This estimated  value of \eqref{eq:lprd} is not an exact lower bound, since this sampling oracle does not provide the guarantee that $\theta$ is indeed feasible in \eqref{eq:sip}.
\item ``Certified lower bound for \eqref{eq:lprd}'' stands for $V_\theta(0,x_0) - \hat{\phi}(\theta)(1+T)$, where $\theta$ is as defined above and $\hat{\phi}(\theta)$ is a guaranteed upper bound on $\phi(\theta)$, the feasibility error of $\theta$ in \eqref{eq:sip}, computed by the global optimization solver \verb+SCIP 8+. As $V_\theta(t,x) + \hat{\phi}(\theta)(t-1-T)  $ is therefore feasible in \eqref{eq:lprd}, the value $V_\theta(0,x_0) - \hat{\phi}(\theta)(1+T)$ is a guaranteed lower bound on $\mathsf{val}$\eqref{eq:lprd}, and, therefore, on $V^*(0,x_0)$.
\item ``Value feedback control \eqref{eq:diffinclusion}'' is the hitting time of the best feasible control generated along the iterations: either with the heuristic control at the first iteration, or the closed-loop controlled trajectory defined by \eqref{eq:diffinclusion} associated with $V = V_{\theta^k}$ at iteration $k$ of Algorithm~\ref{chap:alg:CP}.
\item ``Solution time (in $s$)'' is the total computational time of the heuristic control, of the iterations of Algorithm~\ref{chap:alg:CP} including the sampling oracle, and of the closed-loop trajectory simulation. Therefore, this is the computational time needed to obtain the estimated value of \eqref{eq:lprd} (second column), and the best feasible control (fourth column).
\item ``Iterations number'' is the total number of iterations of Algorithm~\ref{chap:alg:CP}.
\item ``Certification time (in $s$)'' is the computational time of the global optimization solver \verb+SCIP 8+, playing the role of $\delta$-oracle, to compute the aforementioned bound $\hat{\phi}(\theta)$, and deduce the certified lower bound (third column).
\end{itemize}

\subsection{A time-dependent Zermelo problem} We consider a time-dependent nonlinear system with $n = 2$ and $m = 2$, defined by
\begin{align}
\dot{x}_1(t) &= u_1(t) + \frac{1}{2} (1+ t) \: \sin(\pi x_2 (t)) \\
\dot{x}_2(t) &= u_2(t),
\end{align}
with the state constraint set $X = [-1,1] \times [-1,0]$, the control set $U = B(0,1)$. This is the celebrated Zermelo problem, but with a river flow gaining in intensity over time. Fig.~\ref{fig:flow} gives a representation of this flow. The initial condition is $x(0) = (0,-1)$, and the target set is $K = B(0,r)$, for $r = 0.05$. The travel time associated with the heuristic control, consisting in following a straight trajectory, is $1.261$ (see Fig~\ref{fig:zermelocontrol}). Table~\ref{tab:zermelo} presents the numerical results for different values of $d$. We see that the value of the linear semi-infinite program \eqref{eq:lprd} quickly converges as $d$ increases: starting from $d = 6$, the 4 first digits of the estimated value (second column) reach a plateau which corresponds to the value (1.100) of the best feasible trajectory we generate with our feedback control. As regards the certified lower bound, the best value (1.092) is obtained for $d=5$. For greater $d$, we see that increasing $d$ deteriorates the tightness of the best certified bound. This is due to the fact that the separation problem becomes more difficult, with two consequences: (i) the sampling fails to detect unsatisfied constraints, so Algorithm~\ref{chap:alg:CP} stops with a solution that has a real infeasibility $\phi(\theta)$ larger than $\epsilon$ (targeted tolerance), and (ii) the global optimization solver called afterwards does not manage to solve the separation problem to global optimality within the time limit (case $d \in \{7,8\}$), given only a large upper bound $\hat{\phi}(\theta)$ on the true infeasibility $\phi(\theta)$. We notice that as soon as $d \geq 3$, the feedback control defined by \eqref{eq:diffinclusion} (see Sect.~\ref{sec:control}) yields a trajectory that is $13\%$ faster than the heuristic trajectory. In summary, we obtain a certified optimization gap of $0.7\%$ for this minimal time control problem. 
\begin{figure}[ht!]
\centering
\includegraphics[scale=.5]{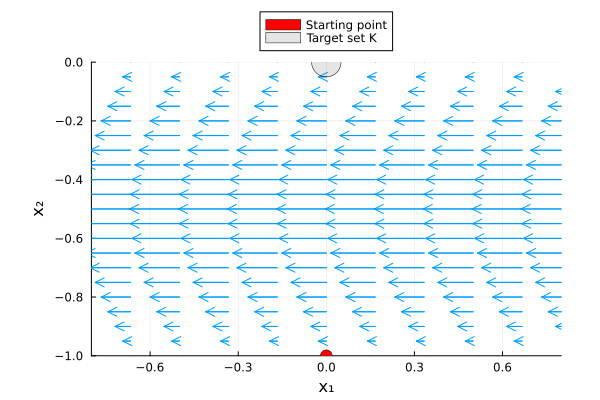}
\caption{Representation of the water flow in the Zermelo problem}
\label{fig:flow}
\end{figure}

\begin{table}[ht!]
\centering
  \begin{tabular}{|c|ccc|ccc|}
    \hline
    \makecell{$d \in \mathbb{N}$} & \makecell{\textbf{Estimated}  \\ \textbf{value of} \eqref{eq:lprd}} & \makecell{\textbf{Certified}  \\ \textbf{LB for} \eqref{eq:lprd}} & \makecell{\textbf{Value feedback} \\ \textbf{control} \eqref{eq:diffinclusion}} & \makecell{\textbf{Solution} \\ \textbf{time (in $s$)}} & \makecell{\textbf{Iterations} \\ \textbf{number}} & \makecell{\textbf{Certification} \\ \textbf{time (in $s$)}} \\\hline    
    2 & 0.952 & 0.945 & 1.261 & 2 & 4 & 1 \\
    3 & 1.064 & 1.044 & 1.101 & 12 & 14 & 12 \\
    4 & 1.096 & 1.084 & 1.100 & 22 & 17 & 1530 \\
    5 & 1.099 & 1.092 & 1.100 & 54 & 22 & 4000 \\
    6 & 1.100 & 1.059 & 1.100 & 60 & 18 & 2330 \\
    7 & 1.100 & 1.051 & 1.100 & 105 & 24 & TL \\
    8 & 1.100 & 0.690 & 1.100 & 215 & 31 & TL \\ \hline
  \end{tabular}
  \caption{Time-dependent Zermelo problem: lower and upper bounds, and computational times for various degrees of the SIP hierarchy \eqref{eq:lprd}}
  \label{tab:zermelo}
\end{table}
\begin{figure}[ht!]
\centering
\includegraphics[scale=.5]{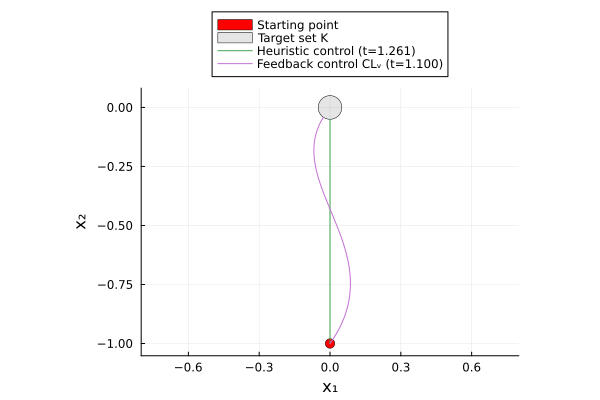}
\caption{Time-dependent Zermelo problem: heuristic control and feedback control ($d=6$)}
\label{fig:zermelocontrol}
\end{figure}

In the special case of this non-polynomial controlled system, a polynomial reformulation exists, at the price of increasing the dimension of the system to $n = 4$:
\begin{align}
\dot{{x}}_1(t) &= u_1(t) + \frac{1}{2} (1+ t) \: \sin(\pi x_2 (t)) \\
\dot{x}_2(t) &= u_2(t) \\
\dot{x}_3(t) &= -\pi x_4(t)u_2(t) \\
\dot{x}_4(t) &=  \pi x_3(t)u_2(t),
\end{align}
with the state constraint set $\hat{X} = [-1,1] \times [-1,0] \times [-1,1] \times [-1,0]  $, the control set $\hat{U} = B(0,1)$, the terminal set $\hat{K} = B(0,r) \times [-1,1] \times [-1,0]$, and the initial condition $x_0 = (0,-1,-1,0)$. The dynamics maintain the equalities $x_3(t) = \cos(\pi x_2(t))$ and $x_4(t) = \sin(\pi x_2(t))$. We are therefore able to compare our approach with the sum-of-squares (SOS) hierarchy, which consists of replacing SIP inequalities in \eqref{eq:lprd} with SOS positivity certificates. For each order $k$ of the hierarchy, i.e., for a maximal degree $d = 2k$ of the polynomial basis, this yields a semi-definite programming problem that we solve with the solver \verb+CSDP+, used with the package \verb+SumOfSquares.jl+. We obtain a polynomial $V(t,x_1,x_2,x_3,x_4)$ that is solution of the corresponding relaxation. Based on this polynomial, we can also generate a feedback controlled trajectory solution of the differential inclusion \eqref{eq:diffinclusion}.
 \begin{table}[ht!]
\centering
  \begin{tabular}{|c|ccc|ccc|}
  \hline
    & \multicolumn{3}{|c|}{\textbf{SIP hierarchy}} &  \multicolumn{3}{|c|}{\textbf{SOS hierarchy}}\\
    \hline
    \makecell{\footnotesize{$d \in \mathbb{N}$}} & \makecell{\textbf{Est./Cert.}  \\ \textbf{LB } } & \makecell{\textbf{Val. feedback} \\ \textbf{control} \eqref{eq:diffinclusion} } & \makecell{\textbf{Sol./Cert.} \\ \textbf{time (in $s$)}} & \makecell{\textbf{Cert.}  \\ \textbf{LB } } & \makecell{\textbf{Val. feedback} \\ \textbf{control} \eqref{eq:diffinclusion} } & \makecell{\textbf{Sol.} \\ \textbf{time (in $s$) }}  \\\hline  
2 & 0.952/0.945 & 1.261& 2/1 & 0.533 & 1.261 & $\leq 1$\\
4 & 1.096/1.084 & 1.101 & 22/1530 & 1.064 & 1.105 & 1 \\
6 & 1.100/1.059 & 1.100 & 60/2330 & 1.099 & 1.100 & 12\\
8 & 1.100/0.690 & 1.100&  215/TL & 1.100 & 1.100 & 190\\
\hline
 \end{tabular}
  \caption{Time-dependent Zermelo problem: comparing the SIP and the SOS approaches}
  \label{tab:zermeloSOS}
\end{table}
Table~\ref{tab:zermeloSOS} compares the performance of the SIP and the SOS approaches. We see that for low-degree polynomials ($d \leq 4$), the semi-infinite hierarchy gives better lower bounds than the SOS hierarchy, although at a higher computational time in the case $d=4$. For $d \in \{6,8 \}$, the lower bound of the SOS hierarchy is tight, while only the estimated lower-bound of the SIP hierarchy is tight: to obtain a certified lower bound, the  SOS hierarchy performs better. For these values of the degree $d$, this optional certification step (calling to the global optimization solver) is costly in the proposed approach. For this first example, where the SOS hierarchy is applicable since a polynomial reformulation of the dynamical system exists, the SIP approach is slower than the SOS hierarchy.  

\subsection{A regatta toy-model} We consider a time-dependent nonlinear (and non-polynomial) system with $n = 2$ and $m = 1$, defined by
\begin{align}
\dot{x}_1(t) &= \mathsf{windspeed}(t) \: \mathsf{polar}\left[ u(t)\right] \: \cos(u(t) + \mathsf{windangle}(t)) \\ 
\dot{x}_2(t) &= \mathsf{windspeed}(t) \: \mathsf{polar}\left[ u(t)\right] \: \sin(u(t)+\mathsf{windangle}(t)),
\end{align}
where $\mathsf{windspeed}(t) = 2 + t$, $\mathsf{windangle}(t) = \frac{\pi}{2}(1-0.4 t)$ and  $\mathsf{polar}[u] = |\sin(\frac{2u}{3})|$. In this model, the control $u(t)$ represents the relative angle between the heading of the boat and the (origin) direction of the wind. The evolution of the wind direction over time is depicted in Fig.~\ref{fig:wind}. The polar curve of this toy model of a sailing boat is represented in Fig~\ref{fig:polar}; this figure clearly shows that this model does not satisfy Assumption~\ref{as:convexF}. Although the absence of duality gap between the control problem and the LP problem \eqref{eq:lpc1} is, therefore, not guaranteed, we see in Table~\ref{tab:regatta} that if this gap exists in this case, it is low (below $1.6 \%$). The state constraint set is $X = [-1,1]^2$, and the control set $U = [-\pi, \pi]$.  The initial condition is $x(0) = (0,-1)$, and the target set is $K = B(0,r)$, for $r = 0.05$. The travel time associated with the heuristic control, consisting in following a straight trajectory, is $1.278$ (see Fig.~\ref{fig:regattacontrol}).
\begin{figure}[!ht]
\centering
     \begin{subfigure}[b]{0.3\textwidth}
         \centering
\includegraphics[width=\textwidth]{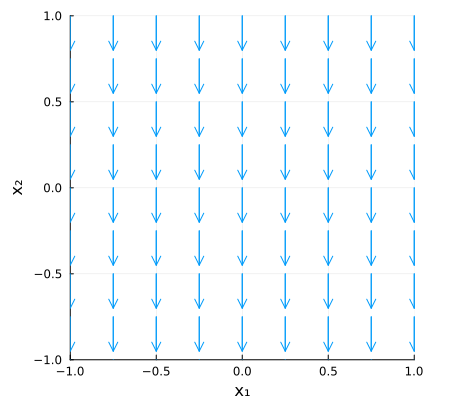}
        \subcaption{$t=0$}
     \end{subfigure}
     \hfill
     \begin{subfigure}[b]{0.3\textwidth}
         \centering
\includegraphics[width=\textwidth]{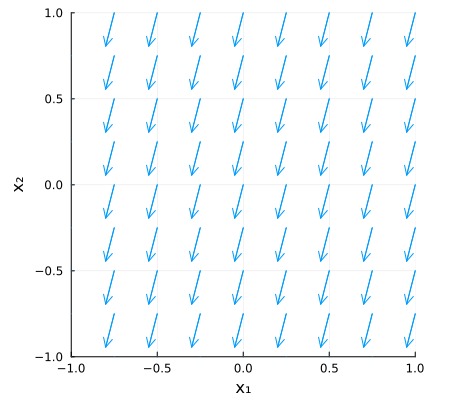}
        \subcaption{$t=0.4$}
     \end{subfigure}
     \hfill
     \begin{subfigure}[b]{0.3\textwidth}
         \centering
\includegraphics[width=\textwidth]{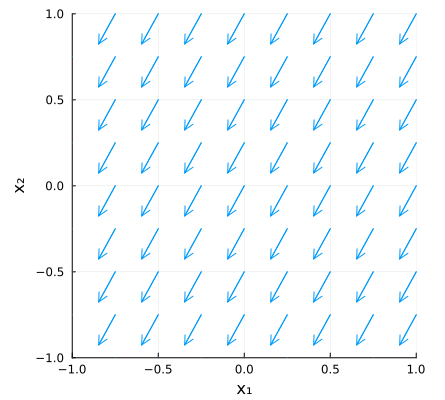}
         \subcaption{$t = 0.8$}
     \end{subfigure}
\caption{Regatta problem: wind direction at different times}
\label{fig:wind}
\end{figure}
\begin{figure}[!ht]
\centering
\includegraphics[scale=.6]{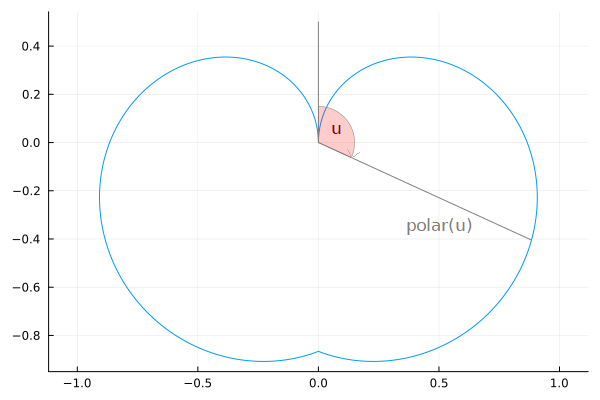}
\caption{Regatta problem: the polar curve of the sailing boat}
\label{fig:polar}
\end{figure}

\begin{table}[ht!]
\centering
  \begin{tabular}{|c|ccc|ccc|}
    \hline
    \makecell{$d \in \mathbb{N}$} & \makecell{\textbf{Estimated}  \\ \textbf{value of} \eqref{eq:lprd}} & \makecell{\textbf{Certified}  \\ \textbf{LB for} \eqref{eq:lprd}} & \makecell{\textbf{Value feedback} \\ \textbf{control} \eqref{eq:diffinclusion}} & \makecell{\textbf{Solution} \\ \textbf{time (in $s$)}} & \makecell{\textbf{Iterations} \\ \textbf{number}} & \makecell{\textbf{Certification} \\ \textbf{time (in $s$)}} \\\hline
   2 & 0.834 & 0.829 & 1.278 & 6 & 6.0 & 2 \\
    3 & 0.896 & 0.880 & 0.912 & 16 & 10.0 & 56 \\
    4 & 0.904 & 0.896 & 0.915 & 31 & 13.0 & 498 \\
    5 & 0.907 & 0.774 & 0.912 & 52 & 16.0 & 1020 \\
    6 & 0.907 & 0.799 & 0.912 & 100 & 22.0 & 1930 \\
    7 & 0.908 & 0.691 & 0.912 & 190 & 29.0 & 7600 \\
    8 & 0.908 & 0.000 & 0.911 & 312 & 33.0 & TL \\\hline
  \end{tabular}
  \caption{Regatta problem: lower and upper bounds, and computational times for various degrees of the SIP hierarchy \eqref{eq:lprd}}
  \label{tab:regatta}
\end{table}

\begin{figure}[ht!]
\centering
\includegraphics[scale=.6]{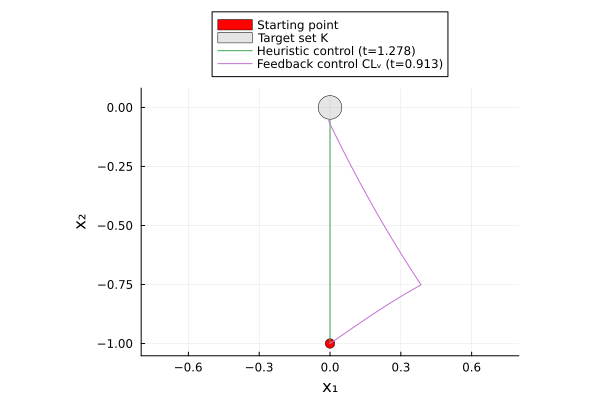}
\caption{Regatta problem: heuristic control and feedback control ($d=6$)}
\label{fig:regattacontrol}
\end{figure}
We see that the highest estimated value of \eqref{eq:lprd}, for $d=7$ and $d=8$, is $0.5\%$ lower than the value ($0.913$) of the best feasible trajectory obtained with the feedback controller for $d=6$. This feedback controller yields a trajectory which is $29\%$ faster than the heuristic trajectory. As regards the certified lower bound, $d = 4$ yields the best result (0.896), at a price of a running time of $498s$ for the exact oracle (\verb+SCIP 8+). For the same reasons as in the previous application, a larger $d$ does not necessarily mean a better certified lower bound obtained within the time limit. In summary, we obtain a certified optimization gap of $1.6\%$ for this minimal time control problem.

\subsection{A generalized Brockett integrator} For $n \in \mathbb{N}^*$ and $m = n-1$ and given a continuous mapping $q\colon \mathbb{R}^n \to \mathbb{R}^m$, we consider the following generalization of the Brockett integrator \cite{loheac_time-optimal_2017}, 
\begin{align}
\dot{x}_i(t) &= u_i(t) \quad \forall i \in \{1, \dots, n-1\} \\
\dot{x}_n(t) &= q(x(t))^\top u(t).
\end{align}
In particular, we study this system for $n = 6$, and $q(x) = \Bigl( 2/(2+x_4),-x_1,-\cos(x_1 x_3),\exp(x_2),x_1 x_2 x_6\Bigr)$. The state constraint set is $X = [-1,1]^5$, and the control set is $U = B(0,1)$. The initial condition is $(x_0) = \frac{1}{2} \mathbf{1}$, and the target set is $K = B(0,r)$, for $r = 0.05$. The travel time associated with the heuristic control is 1.377.

\begin{table}[ht!]
\centering
  \begin{tabular}{|c|ccc|ccc|}
  \hline
    \makecell{$d \in \mathbb{N}$} & \makecell{\textbf{Estimated}  \\ \textbf{value of} \eqref{eq:lprd}} & \makecell{\textbf{Certified}  \\ \textbf{LB for} \eqref{eq:lprd}} & \makecell{\textbf{Value feedback} \\ \textbf{control} \eqref{eq:diffinclusion}} & \makecell{\textbf{Solution} \\ \textbf{time (in $s$)}} & \makecell{\textbf{Iterations} \\ \textbf{number}} & \makecell{\textbf{Certification} \\ \textbf{time (in $s$)}} \\\hline
     2 & 1.071 & 0.763 & 1.071 & 220 & 28 & 73 \\
    3 & 1.072 & 0.000 & 1.071 & 5630 & 133 & TL \\
    4 & 1.072  & 0.000  & 1.070 & 147400  & 319 & TL \\ \hline    
  \end{tabular}
  \caption{Generalized Brockett integrator: lower and upper bounds, and computational times for various degrees of the SIP hierarchy \eqref{eq:lprd}}
  \label{tab:brockett}
\end{table}
Since this system has a larger dimension than the other two examples, we see that the computation times are longer for the same degree $d$. Already for $d=2$, we obtain an estimated value of \eqref{eq:lprd} that is within $0.1\% $ of the value of the feedback control (1.070). This feedback control yields an improvement of $22\% $ over the heuristic trajectory. Note that the estimated values of \eqref{eq:lprd} computed by Algorithm~\ref{chap:alg:CP} with the (inexact) sampling oracle are slightly larger than the value of the best trajectory we computed: thus, these estimates are not valid lower bounds, but only estimates of the value of the minimum time control problem. Regarding the certification of lower bounds, the global optimization solver \verb+SCIP 8+ fails to produce tight upper and lower bounds on $\phi(\theta)$, the infeasibility of the solution $\theta$ returned by Algorithm~\ref{chap:alg:CP}. Therefore, the resulting certified lower bounds are not tight either. In summary, we obtain a certified optimization gap of $29\%$ for this minimal time control problem.

\section{Discussion} We apply the dual approach in minimal time control, that consists in searching for maximal subsolutions of the HJB equation, to generic nonlinear, even non-polynomial, controlled systems. The basis functions used to generate these subsolutions are polynomials, that are subject to semi-infinite constraints. We prove the theoretical convergence of the resulting hierarchy of semi-infinite linear programs, and our numerical tests on three different systems show good convergence properties in practice. These results show that the use of a random sampling oracle allows a good approximation of the value of the control problem. For small systems, it is even possible to obtain tight and certified lower bounds, based on a global optimization solver. Finally, the numerical experiments also show that the computed subsolutions of the HJB equations help to recover near-optimal controls in a closed-loop form. As illustrated in these numerical experiments, the advantage of our approach based on semi-infinite programming, compared to the sum-of-squares approach, is the ability to handle non-polynomial systems. In the numerical example where a polynomial reformulation of the system was possible, the sum-of-squares approach was, however, faster.

A promising avenue for continuing this work is to investigate the use of a other basis of functions to search for an approximate value function, resulting in other semi-infinite programming hierarchies with convergence guarantees. In particular, it would be relevant to use non-differentiable functions in the basis to improve approximation capabilities for non-differentiable value functions. Another avenue of research is to extend the approach and theoretical results to a generic optimal control problem.

\section*{Acknowledgments}
The authors would like to thank Maxime Dupuy, Leo Liberti and Claudia D'Ambrosio for fruitful discussions and advice.

\bibliography{bibliolatex}
\bibliographystyle{plain}

\appendix

\section{Technical lemmata}

\begin{lemma}
We consider a compact set $Z \subset \mathbb{R}^p$, and the family of compact sets $Z_\delta = \{ z \in  \mathbb{R}^p \: \colon \: d(z,Z)  \leq \delta \}$ for any $\delta \geq 0$ and a continuous function $\psi \in C(\mathbb{R}^p)$. Then, the function $\Psi(\delta) = \min_{z \in Z_\delta} \psi(z)$ is continuous at $0$.
\label{lem:valuefunction}
\end{lemma}
\begin{proof}
First of all, we notice that the function $\delta \mapsto \min_{z \in Z_\delta} \psi(z)$ is well-defined, since $\psi$ is continuous and $Z_\delta$ is compact. As $Z_{\delta_1} \subset Z_{\delta_2}$ for any $\delta_1 \leq \delta_2$, the function $\Psi$ is non-increasing, which proves that the following limit exists: 
\begin{align} 
\underset{\delta \rightarrow 0^+}{\lim} \Psi(\delta) = \Psi(0^+) \leq \Psi(0).
\label{eq:limitright}
\end{align} 
We take a positive sequence $(\delta_k) \in \mathbb{R}_{++}^\mathbb{N}$ such that $\delta_k \rightarrow 0$. Hence, $\Psi(\delta_k) \rightarrow \Psi(0^+)$ by  definition of the right-limit. For any $k \in \mathbb{N}$, we define $z_k \in Z_{\delta_k}$ such that $\psi(z_k) = \Psi(\delta_k)$. The sequence $(\delta_k)$ being bounded, we can introduce an upper bound $\Bar{\delta}$. Hence, any element of the sequence $(z_n)$ belongs to the compact set $Z_{\Bar{\delta}}$, and up to the extraction of a subsequence, converges to  a point $z$ being such that $\psi(z) = \Psi(0^+)$ by continuity of $\psi$ and uniqueness of the limit. As $d(z_k,Z)$, the distance between $z_k$ and the compact set $Z$, is bounded above by $\delta_k$ and is non-negative, it converges to $0$. By continuity of the distance, we know that $d(z,Z) = 0$ and, thus, $\Psi(0^+) = \psi(z) \geq  \Psi(0)$. Together with Eq.~\eqref{eq:limitright}, this yields $\Psi(0^+) =   \Psi(0)$.
\end{proof}
\begin{lemma}
Let $Q \in C^1(\mathbb{R}^{p})$ be a continuously differentiable function, with a locally Lipschitz gradient. Let $Z \subset \mathbb{R}^{p}$ be a compact set. Then, there exists a constant $A > 0$, and a sequence of polynomials $(w_d(x))_{d \in \mathbb{N}^*}$ such that for all $d \in \mathbb{N}^*$, $w_d \in \mathbb{R}_d[x_1,\dots,x_p]$ and 
\begin{align}
\sup_{x \in Z} | w_d (x) -  Q (x) | \leq \frac{A}{d}  \\ 
\sup_{x \in Z} \lVert \nabla w_d (x) -  \nabla Q (x) \lVert \leq \frac{A}{d}.
\end{align}
\label{lem:polapprox}
\end{lemma}
We underline that the constant $A$ implicitly depends on $p$, $Q$ and $Z$, but not on the polynomial $w_d(x)$, nor on its degree $d$.
\begin{proof} We introduce a constant $R > 0$ such that $Z \subset B(0,R)$, and the function $\tilde{\mol} = \mol \ast \mathbf{1}_{B(0,R+1)}$, where $\mol$ is the mollifier introduced in the proof of Th.~\ref{th:smoothvf}. We notice that $\tilde{\mol} \in C^\infty(\mathbb{R}^p)$ is supported on $\tilde{Z} = B(0, R +2)$ and constant equal to $1$ over  $B(0,R)$. We define $\tilde{Q}(x) = Q(x) \tilde{\mol}(x)$, and we notice that (i) $\tilde{Q}$ is supported on the compact set $\tilde{Z}$, which contains $Z$, (ii) for all $x \in Z$, $\tilde{Q}(x) = Q(x)$ and $\nabla \tilde{Q}(x) = \nabla Q(x)$. Applying \cite[Th.~1]{bagby_multivariate_2002} to the function $\tilde{Q}$, that has a compact support, we know that there exists a constant $C$ such that for any $d \geq 1$, there exists a polynomial $w_d(x)$ of degree at most $d$ such that 
\begin{align}
\label{eq:bounding1}
\sup_{x \in Z} |w_d(x) - \tilde{Q}(x) | \leq \frac{C}{d} \kappa(\frac{1}{d}) \leq C \kappa(\frac{1}{d}), \\
\sup_{x \in Z} |\partial_i(w_d - \tilde{Q})(x) | \leq C \kappa(\frac{1}{d})
\label{eq:bounding2}
\end{align}
where $\kappa(\delta) = \underset{1 \leq i \leq p }{\sup} \:  \underset{\begin{subarray}{c} (x,y) \in  \mathbb{R}^p \times \mathbb{R}^p \\ |x-y|\leq \delta \end{subarray}}{\sup} |\partial_i \tilde{Q}(x) - \partial_i \tilde{Q}(y)|$. We define $\tilde{Z} = \{ x \in \mathbb{R}^p \colon d(x,Z) \leq 1 \}$. Since $\partial \tilde{Q}$ is uniformly null outside $\tilde{Z}$, and assuming that $\delta \leq 1$, we notice that 
\begin{align*}
\kappa(\delta) = \underset{1 \leq i \leq p }{\sup} \:  \underset{\begin{subarray}{c} x,y \in \tilde{Z} \times \mathbb{R}^p \\ |x-y|\leq \delta \end{subarray}}{\sup} |\partial_i \tilde{Q}(x) - \partial_i \tilde{Q}(y)| = \underset{1 \leq i \leq p }{\sup} \:  \underset{\begin{subarray}{c} x,y \in \tilde{Z} \times \tilde{Z} \\ |x-y|\leq \delta \end{subarray}}{\sup} |\partial_i \tilde{Q}(x) - \partial_i \tilde{Q}(y)|.
\end{align*}
We note that $\partial_i \tilde{Q}(x) =\tilde{\mol}(x)  \partial_i {Q} (x) + Q(x)  \partial_i {\tilde{\mol}} (x)$, and therefore, $|\partial_i \tilde{Q}(x) - \partial_i \tilde{Q}(y)| = | \tilde{\mol}(x)  (\partial_i {Q} (x) - \partial_i {Q} (y)) + \partial_i {Q} (y)(\tilde{\mol}(x) - \tilde{\mol}(y)) + Q(x)  (\partial_i {\tilde{\mol}} (x) - \partial_i {\tilde{\mol}} (y)) + \partial_i {\tilde{\mol}} (y)(Q(x) - Q(y))|$. We use, then, the triangle inequality and the facts that (i) $\tilde{\mol}$ is $C^\infty$, therefore bounded, Lipschitz continuous, and with a Lipschitz-continuous gradient over $\tilde{Z}$ and (ii) $Q$ is continuously differentiable, therefore bounded, and  Lipschitz continuous over $\tilde{Z}$; by assumption it has a Lipschitz continuous gradient over the compact set $\tilde{Z}$. We deduce that $\partial_i \tilde{Q}$ is Lipschitz continuous over $\tilde{Z}$: there exists  $L_i>0$ such that  $\underset{\begin{subarray}{c} x,y \in \tilde{Z} \times \tilde{Z} \\ |x-y|\leq \delta \end{subarray}}{\sup} |\partial_i \tilde{Q}(x) - \partial_i \tilde{Q}(y)| \leq L_i \delta$, for all $\delta \in [0, 1]$. Defining $L = \max_i L_i$, we deduce  $\kappa(\delta) \leq L \delta$. Then Eq.~\eqref{eq:bounding1} reads $\sup_{x \in Z} |w_d(x) - \tilde{Q}(x) | \leq \frac{CL}{d}$, and  Eq.~\eqref{eq:bounding2} reads $\sup_{x \in Z} |\partial_i (w_d - \tilde{Q})(x) | \leq \frac{CL}{d}$ for all $i \in \{1, \dots, p \}$. We also deduce that $\sup_{x \in Z} \lVert \nabla w_d (x) -  \nabla \tilde{Q} (x) \lVert \leq \frac{C L p}{d}$. Defining $A = C L p$, and noticing that 
for all $x \in Z$, $\tilde{Q}(x) = Q(x)$ and $\nabla \tilde{Q}(x) = \nabla Q(x)$, one obtains the claimed statement.
\end{proof}

\begin{lemma}
Under  Assumption~\ref{as:convexF} and Assumption~\ref{as:finite}, we consider an admissible trajectory $(x(\cdot),u(\cdot))$ over $[0, t_1]$ of the minimal time control problem \eqref{eq:system}-\eqref{eq:inf} starting from $(0,x_0)$. Then, for almost all $t \in [0, t_1]$, $f(t,x(t),u(t)) \in T_X(x(t))$.
\label{lem:tangentcone}
\end{lemma}
\begin{proof}
We reduce to a time-invariant controlled system: we define, for any $y = (t,x) \in \mathbb{R}^{n+1}$ and $u \in \mathbb{R}^m$, $\tilde{f}(y,u) = \begin{pmatrix}
1 \\ f(y,u)
\end{pmatrix}$, and the constant set-valued map $\tilde{U}(y) = U$. The control system $(\tilde{f},\tilde{U})$ is a Marchaud control system \cite[Def.~6.1.3]{aubin_jean-pierre_viability_1991}, since: (i) the graph of $\tilde{U}$ is closed (ii) $\tilde{f}$ is continuous (iii) the velocity subsets $\{ \tilde{f}(y,u) \colon u \in \tilde{U}(y) \}$ are convex due to Assumption~\ref{as:convexF}, and (iv) the function $f$ has a linear growth since it is Lipschitz continuous, and the set-valued map are bounded, thus also has a linear growth. We define the set $\mathcal{C} = \mathbb{R}_+ \times X$ and notice that the control $u(\cdot)$ regulates a trajectory $y(t) = \begin{pmatrix}
t \\ x(t)
\end{pmatrix}$ that remains in $\mathcal{C}$, therefore according to \cite[Th.~6.1.4]{aubin_jean-pierre_viability_1991}, for all most all $t \in [0, t_1]$, $u(t) \in \{ u \in \tilde{U}(y(t)) \colon \tilde{f}(y(t),u(t)) \in T_\mathcal{C}(y(t)) \}$. We notice that $ T_\mathcal{C}(y(t)) \subset \mathbb{R} \times T_X(x(t))$, implying that $f(y(t),u(t)) \in T_X(x(t))$.
\end{proof}

\begin{lemma}
For any locally Lipschitz continuous function $F \colon \mathbb{R}^p \to \mathbb{R}$, and for any Lipschitz continuous curve $y \colon [0, T] \to \mathbb{R}^p$, the function $t \mapsto F(y(t))$ is differentiable a.e. and satisfies
\begin{align}
\min_{g \in \partial^c F(y(t))} g^\top \dot{y}(t) \leq \frac{d}{dt}(F(y(t))) \leq \max_{g \in \partial^c F(y(t))} g^\top \dot{y}(t),  \label{eq:chainrulineq}
\end{align} 
for almost all $t \in [0, T]$.
\label{lem:clarkedifferential}
\end{lemma}
The particular functions $F$ for which these three quantities are equal are called path-differentiable in \cite{bolte_conservative_2021}.

\begin{proof}
First, we notice that the functions $t \mapsto y(t)$ and $t \mapsto F(y(t))$ are Lipschitz continuous, therefore differentiable a.e. on $[0, T]$ due to the Rademacher theorem \cite{gariepy_measure_2015}. Hence, for almost all $t \in [0, T]$, both $y(t)$ and $F(y(t))$ are differentiable at $t$. We consider such a $t$, and we show that \eqref{eq:chainrulineq} holds for this particular $t$, which we prove the Lemma. Since $y$ is differentiable at $t$, $r(h) = y(t+h) -  y(t) - h \dot{y}(t)$ is in $o_{h\to 0}(h)$ . 

 Since $s \mapsto F(y(s))$ is differentiable at $t$, the following holds
\begin{align}
\frac{d}{dt}(F(y(t))) &= \underset{h \rightarrow 0, h >0}{\lim} \frac{F(y(t+h)) - F(y(t))}{h} \\
		&=  \underset{h \rightarrow 0, h >0}{\lim} \frac{F(y(t) + h \dot{y}(t) + r(h)) - F(y(t))}{h} \label{eq:limitderiv}
\end{align}
Since $r(h) = o_{h\to 0}(h)$ and $F$ is locally Lipschitz, we know that $\underset{h \rightarrow 0, h >0}{\lim} \frac{F(y(t)) - F(y(t) + r(h)) }{h} = 0$. Summing this with Eq.~\eqref{eq:limitderiv}, we deduce that

\begin{align}
\frac{d}{dt}(F(y(t))) &= \underset{h \rightarrow 0, h >0}{\lim} \frac{F(y(t) + r(h) + h \dot{y}(t) ) - F(y(t) + r(h))}{h} \\
& \leq \underset{\begin{subarray}{c} y' \rightarrow y(t) \\ h \rightarrow 0, h> 0 \end{subarray}}{\limsup} \frac{F(y' + h \dot{y}(t)) - F(y')}{h} = F^\circ(y(t), \dot{y}(t)),
\end{align}
where  $F^\circ(y;v) = \underset{\begin{subarray}{c} y' \rightarrow y \\ h \rightarrow 0, h> 0 \end{subarray}}{\limsup} \frac{F(y' + h v) - F(y')}{h}$ is the $F^\circ(y;h)$ Clarke's directional derivative at $y \in \mathbb{R}^p$ in the direction $v \in \mathbb{R}^p$. The inequality follows from the fact that $y(t) +r(h) \rightarrow y(t)$. By property of the Clarke subdifferential \cite{clarke_generalized_1975}, we also know that $F^\circ(y;v) = \max_{g \in \partial^c F(y)} g^\top v$. Hence, in particular,
\begin{align}
\frac{d}{dt}(F(y(t)))  \leq \max_{g \in \partial^c F(y(t))} g^\top \dot{y}(t). \label{eq:ineqclarkemax}
\end{align}
The reasoning that proved Eq.~\eqref{eq:ineqclarkemax} is also applicable to $-F$, that is also locally Lipschitz, and such that $s \mapsto (-F)(s)$ is differentiable at $t$. Therefore,
\begin{align}
\frac{d}{dt}(-F(y(t)))  \leq \max_{g \in \partial^c (-F)(y(t))} g^\top \dot{y}(t).
\end{align}
As $\partial^c (-F)(y(t)) = - \partial^c F(y(t))$ by property of the Clarke subdifferential, we deduce that 
\begin{align}
-\frac{d}{dt}(F(y(t)))  \leq \max_{g \in \partial^c F(y(t))} - g^\top \dot{y}(t) = - \min_{g \in \partial^c F(y(t))} g^\top \dot{y}(t),
\end{align}
and therefore $\frac{d}{dt}(F(y(t)))  \geq \min_{g \in \partial^c F(y(t))} g^\top \dot{y}(t)$.
\end{proof}

\end{document}